\documentclass[twoside,11pt]{article}

\usepackage[english]{babel}
\usepackage{amsmath}
\usepackage{amsthm}
\usepackage{amsfonts}
\usepackage{amssymb}
\usepackage{graphicx}
\usepackage{colortbl,dcolumn}
\usepackage{xcolor} 
\newcommand{\fix}[1]{{\color{red}#1}}

\newenvironment{fixblock}{\begingroup\color{red}}{\endgroup}

\usepackage{paralist}  

\allowdisplaybreaks[4]
       \newtheorem{lemma}{\bf Lemma}[section]
       \newtheorem{theorem}{\bf Theorem}[section]
       \newtheorem{proposition}{\bf Proposition}[section]

       \newtheorem{remark}{\bf Remark}[section]

       \numberwithin{equation}{section}

\newcommand{\R}{\mathbb{R}}

\newcommand{\be}{\begin{equation}}
\newcommand{\ee}{\end{equation}}

\begin{document}
\title{{\Large Asymptotic stability of strong rarefaction waves to a parabolic-hyperbolic system arising from chemotaxis}
\footnotetext{\small *Corresponding author.}
\footnotetext{\small E-mail addresses: liust533@nenu.edu.cn (S. Liu), lijy645@nenu.edu.cn (J. Li)}
}
\author{{Sitong Liu, Jingyu Li$^{*}$}\\[2mm]
\small\it School of Mathematics and Statistics, Northeast Normal University,\\
\small\it   Changchun 130024, PR China }
\date{ }
\maketitle

\begin{quote}
\small \textbf{Abstract:}
We are interested in the asymptotic behavior of solutions toward strong rarefaction waves for a parabolic-hyperbolic system arising from chemotaxis. Suppose that the Riemann problem to the corresponding inviscid system admits rarefaction waves. We show that if the initial data is a small perturbation of an approximate  rarefaction wave, then the Cauchy problem has a unique global solution that converges to the rarefaction wave asymptotically in time. The stability results hold regardless of the wave strengths. The waves can be either a single  rarefaction wave or a composite of two  rarefaction waves. In the composite case, we separate the initial centers of the two approximate profiles to make their interaction error term small while retaining convergence to the prescribed Riemann wave. The proofs are based on the energy method, where the key observations are the monotonicity of the approximate rarefaction waves with respect to both space and time, and the appropriate replacement of spatial derivative of the hyperbolic component with the parabolic component in proper forms.

\indent \textbf{Key words}: Chemotaxis; parabolic-hyperbolic system; strong rarefaction wave; composite wave; asymptotic stability.

\indent \textbf{AMS(2020) Subject Classification}: 35B35, 35B40, 35Q92, 92C17

\end{quote}

\section{Introduction}\label{Sect.1}
In this paper, we are concerned with the large time behaviors of solutions to the Cauchy problem of the following parabolic-hyperbolic system:
\begin{equation}\label{1.1}
\begin{cases}
n_{t} -(nq)_{x}=n_{xx}, &\ x\in\R,t>0,\\
q_{t}-n_{x}=0, &\ x\in\R,t>0,
\end{cases}
\end{equation}\\
with the initial data
\begin{equation}\label{i d}
(n,q)(x,0)=(n_0,q_0)(x)\rightarrow (n_\pm ,q_\pm ) \ \text{ as } x\rightarrow \pm\infty.
\end{equation}
The system \eqref{1.1} is derived from the following PDE-ODE hybrid chemotaxis model with logarithmic sensitivity
\begin{equation}\label{chemotaxis model}
\left\{
\begin{aligned}
&n_t=D n_{xx}-\chi(n(\ln c)_x)_x,\\
&c_t=- nc+\beta c,
\end{aligned}
\right.
\end{equation}
which was first proposed by Levine et al. \cite{H.A. Levine 00} to model the interactions between vascular endothelial cells and the signaling molecule vascular endothelial growth factor (VEGF) during the initiation of tumor angiogenesis. In this system, the unknowns $n(x,t)>0$ and $c(x, t)>0$ denote the density of vascular endothelial cells and the concentration of VEGF, respectively. In contrast to random diffusion, chemotaxis represents a biased movement of organisms in response to chemical stimuli. The parameter $D>0$ denotes the diffusion coefficient of vascular endothelial cells. The parameter $\chi>0$ stands for the chemotactic coefficient measuring the strength of chemotaxis, and the parameter $\beta\geq0$ denotes the growth rate of chemical VEGF. The logarithmic sensitivity function $\ln c$ of the model \eqref{chemotaxis model} reflects that the chemotactic response of cells to the chemical signal adheres to the Weber-Fechner law, a principle with prominent  implications in biological modelings (cf. \cite{W. Alt 87, F.W. Dahlquist 72, E.F. Keller 71}). Since the sensitivity function $\ln c$ is singular at $c=0$, it is challenging to study the model \eqref{chemotaxis model} directly. To resolve this singularity, Levine and Sleeman \cite{H.A. Levine 97} introduced an effective Hopf-Cole transformation
\begin{equation*}
\begin{aligned}
q:=-(\ln c)_x=-\frac{c_x}{c},
\end{aligned}
\end{equation*}
which transform the chemotaxis model \eqref{chemotaxis model} to the following parabolic-hyperbolic system
\begin{equation*}
\left\{
\begin{aligned}
&n_t-\chi(nq)_x=D n_{xx},\\
&q_t-n_x=0.
\end{aligned}
\right.
\end{equation*}
Then applying the scaling transformations
\begin{equation*}
\begin{aligned}
\tilde{x}=\frac{\sqrt{\chi}}{D}x,\qquad
\tilde{t}=\frac{\chi}{D}t,\qquad
\tilde{q}=\sqrt{\chi}q,
\end{aligned}
\end{equation*}
and leaving $n$ unchanged, we obtain \eqref{1.1} after dropping the tildes, for arbitrary $D,\chi>0$.

The first analytical work on the well-posedness of the system \eqref{1.1}-\eqref{i d} was established by Guo et al. \cite{GXZZ}, where it was shown that the system admits a unique global strong solution for large initial data. Subsequently, by taking a positive constant state as the background solution, Zhang et al. \cite{Y. Zhang 13} constructed global classical small solutions and obtained an algebraic convergence rate. By establishing an entropy type estimate relative to a constant state, Li et al. \cite{D.Li 15} further proved the global well-posedness of large classical solutions. In the multi-dimensional case, owing to the complexity of higher dimensionality, all existing results on global well-posedness have been established for small solutions around a constant state. See \cite{D.Li 15,DL,Hao} for results concerning global well-posedness in various working spaces.

All the above works are concerned with the stability of constant steady states.  Once $n_+\neq n_-$, it  turns out that the large time behaviors of solutions to \eqref{1.1} are intrinsically related to the Riemann problem for inviscid system:
\begin{equation}\label{conservation law}
\left\{
\begin{aligned}
&n_{t} -(nq)_{x}=0,\\
&q_{t}-n_{x}=0,
\end{aligned}
\right.
\end{equation}
with initial data
\begin{equation}\label{n0,q0}
(n,q)(x,0)=(n_0^r,q_0^r)(x)=
\left\{
\begin{aligned}
&(n_-,q_-),\ \ \ \ x<0,\\
&(n_+,q_+),\ \ \ \ x>0.
\end{aligned}
\right.
\end{equation}
To analyze its mathematical structure, we rewrite \eqref{conservation law} as
\begin{equation}\label{equivalent system}
\left(\begin{aligned}
&n\\&q
\end{aligned}
\right)_t
+\left(\begin{aligned}
&-q\ \ -n\\&-1\ \ \ \ \ 0
\end{aligned}\right)
\left(\begin{aligned}
&n\\&q
\end{aligned}
\right)_x=0.
\end{equation}
A direct calculation shows that the matrix
\begin{equation*}
J=\left(\begin{aligned}
&-q\ \ -n\\&-1\ \ \ \ \ 0
\end{aligned}\right)
\end{equation*}
has two distinct eigenvalues
\begin{equation}\label{1.8}
\lambda_1(n,q)=\frac{-q-\sqrt{q^2+4n}}{2}
\ <\ \frac{-q+\sqrt{q^2+4n}}{2}=\lambda_2(n,q),
\end{equation}
if $q^2+4n>0$, and the corresponding right eigenvectors are given by
\begin{equation*}
r_k(n,q)=(-\lambda_k(n,q), 1)^T,\
k=1,2.
\end{equation*}
It is easy to verify that
\begin{equation}\label{genuinely nonlinear}
\begin{aligned}
&\nabla \lambda_1(n,q)\cdot r_1(n,q)=
-\frac{q}{\sqrt{q^{2}+4n}}-1\neq 0,\\
&\nabla \lambda_2(n,q)\cdot r_2(n,q)=
\frac{q}{\sqrt{q^{2}+4n}}-1\neq 0,
\end{aligned}
\end{equation}
where $\nabla \lambda_k(n,q)\triangleq(\frac{\partial\lambda_k}{\partial n}, \frac{\partial\lambda_k}{\partial q})^T ,k=1,2$.
\fix{Thus the system \eqref{conservation law} is strictly hyperbolic when $q^2+4n>0$. Both characteristic fields are genuinely nonlinear in the positive-density region $n>0$ used below. In a strictly hyperbolic, genuinely nonlinear region, admissible local Riemann solutions can be assembled from elementary shock and rarefaction curves \cite{J. A. 83}. However, the global Riemann structure on the larger region $n\geq0$ is more subtle, our analysis is confined to the positive-density rarefaction branches and the admissible two-rarefaction region studied in \cite{Rascle,T.Li 22,F.He 24}.}

{\color{red}
For the system \eqref{1.1}, the stability theory of traveling waves is comparatively well developed.}  Wang and Hillen \cite{Wang08} were the first to construct shock waves for the inviscid system \eqref{conservation law} and explicitly derived the traveling waves for the parabolic-hyperbolic system \eqref{1.1}. Subsequently, using the energy method, Li and Wang \cite{T.Li 09} proved the asymptotic stability of traveling waves for \eqref{1.1} away from the vacuum state under zero-mass perturbations. Li et al. \cite{J.Li 13} generalized the work of \cite{T.Li 09} to composite two traveling waves. On the basis of weighted energy estimates, Jin et al.  \cite{Jin 13} further derived the asymptotic stability of traveling waves with vacuum end state under zero mass perturbations. The first multidimensional study was conducted by Chae et al. \cite{Chae}, who obtained the stability of planar traveling waves on an infinite strip domain under zero-mass perturbations. All these works were restricted to zero-mass perturbations as their methods essentially rely on the anti-derivative approach. Very recently, a breakthrough was achieved by Choi et al. \cite{Choi,Choi 20} where the so-called $L^2$ contraction property was successfully derived for general perturbations of traveling waves via the relative entropy method. Based on such relative entropy method, along with some intrinsic cancellation mechanisms of the chemotaxis model, the authors of the current paper \cite{Liu} proved the nonlinear stability of planar  traveling waves in the three-dimensional case under general perturbations.

Compared with the fruitful works available for traveling waves of \eqref{1.1}, the understanding of rarefaction waves is quite limited. To our knowledge, Rascle \cite{Rascle} was the first to construct rarefaction waves for the inviscid system \eqref{conservation law}, where a strictly hyperbolic region and a linearly degenerate region were found. Very recently, Li and Mathur \cite{T.Li 22} as well as He and Wang \cite{F.He 24} presented a comprehensive study of the Riemann problem for \eqref{conservation law} and constructed rarefaction waves, shock waves, contact discontinuities, and their superpositions. However, it remains an open question whether these basic waves are stable under appropriate perturbations. Furthermore, for the parabolic-hyperbolic system \eqref{1.1}, it is unknown whether its large-time behaviors could include a single rarefaction wave and its superposition with rarefaction waves or shock waves from other fields.

{\color{red}
Early contributions on the stability of rarefaction wave include the studies of Matsumura and Nishihara \cite{Matsumura 86} and Liu and Xin \cite{Liu-Xin}, as well as Xin's general stability theory for weak rarefaction waves, including the superposition of two rarefaction waves \cite{Z.P Xin 88,Z.P Xin 89}. Subsequently, Matsumura and Nishihara \cite{Matsumura 92} established the stability of strong rarefaction waves for the barotropic Navier-Stokes equations with  $1\leq\gamma\leq2$. Further advances treated strong waves for compressible Navier-Stokes equations \cite{Nishihara 04}, large initial perturbations \cite{R. Duan 09}, the Navier-Stokes-Poisson system \cite{Duan-Liu}, and planar rarefaction waves to the multidimenional Navier-Stokes equations  \cite{Li-Wang,Li-Wang-Wang}.

Composite patterns present the additional difficulty of controlling interactions between waves with different structures and decay rates. For the compressible Navier-Stokes equations, Huang et al. \cite{HLM10} used energy estimates to establish the stability of a viscous contact wave combined with rarefaction waves. Kang et al. employed relative entropy and $a$-contraction with shifts for shock-rarefaction composites \cite{KVW23}, and extended the analysis to generic Riemann wave patterns for the compressible Navier-Stokes-Fourier equations \cite{KVW25}. Related results concern boundary-layer-rarefaction-shock patterns \cite{HKK26} and shock-rarefaction composites for the Navier-Stokes-Poisson system \cite{Shi26}. A different approach to the stability of rarefaction waves, based on pointwise Green function estimates, was developed by Szepessy and Zumbrun \cite{SZ96}. These studies highlight the importance of developing stability estimates adapted to the dissipative structure of the underlying system and the interactions among distinct wave components.
}

The purpose of this paper is to prove the asymptotic stability of rarefaction waves and the superposition of two rarefaction waves, without assuming that the strengths of the prescribed rarefaction waves are small. {\color{red}Our construction
of smooth approximate rarefaction waves follows the Burgers approximation method employed in
 \cite{Matsumura 92}, while the two-mode construction is also closely related to the framework developed in \cite{Z.P Xin 89}. The main contribution of the present paper is to establish the stability of strong rarefaction waves
for the partially dissipative chemotaxis system \eqref{1.1}, whose degenerate viscosity structure and
nonlinear coupling require system-specific energy estimates. In particular, our analysis exploits
the parametrization of rarefaction curves by Riemann invariants and the monotonicity properties
of the approximate waves, together with suitable estimates for the interactions between two
strong approximate rarefaction waves.
}

We now state our main results. {\color{red}Let $\Omega:=\{(n,q)\in\mathbb R^2:n>0\}$ and fix $(n_-,q_-)\in\Omega$. For $k=1,2$, let $\Gamma_k(n_-,q_-)$ be the maximal connected integral curve in $\Omega$ through $(n_-,q_-)$ of $(n',q')(s)=r_k(n,q)(s)$. We define the corresponding rarefaction branch by
\begin{equation}\label{Rk}
\begin{aligned}
&R_k(n_-,q_-):=\{(n,q)\in\Gamma_k(n_-,q_-):\lambda_k(n,q)\geq\lambda_k(n_-,q_-)\},
\end{aligned}
\end{equation}
This definition selects the rarefaction branch on the integral curve through $(n_-,q_-)$. All wave profiles below take values in fixed compact arcs of this branch, with density bounded away from zero. Following \cite{F.He 24,T.Li 22}, we use the following $k$-Riemann invariants $h_k$, each of which is constant along the corresponding integral curves:}
\begin{equation}\label{RI}
\begin{aligned}
&h_1(n,q)=(\sqrt{q^2+4n}+q)(\sqrt{q^2+4n}-2q)^2,\\
&h_2(n,q)=(\sqrt{q^2+4n}-q)(\sqrt{q^2+4n}+2q)^2.
\end{aligned}
\end{equation} It is easy to verify that the Riemann invariants satisfy
\begin{equation}\label{RId}
\begin{aligned}
& \nabla h_k\cdot r_k=0,\
k=1,2.
\end{aligned}
\end{equation}
\fix{Consequently, $h_k(n,q)=h_k(n_-,q_-)$ for $(n,q)\in\Gamma_k(n_-,q_-)$. This level-set identity is always used together with the specified integral curve and the branch condition in \eqref{Rk}.}

Our first result is about the asymptotic stability of single rarefaction waves.

\begin{theorem}\label{the1.1}
{\color{red}Fix $\theta>\frac32$, $k\in\{1,2\}$, and an admissible pair of end states $(n_-,q_-)$ and $(n_+,q_+)\in R_k(n_-,q_-)$  with $n_\pm>0$. 
There exist constants $\varepsilon_0,\delta_0>0$, depending on the fixed end states and $\theta$ such that, for every $0<\varepsilon\leq\varepsilon_0$, }
if 
\begin{equation}\label{n0-n0r}
\|n_{0}-N_{k0}, q_{0}-Q_{k0}\|_{H^1}\leq\delta_0,
\end{equation}
then the initial value problem \eqref{1.1}-\eqref{i d} has a unique global solution $(n,q)(x,t)$ satisfying
\begin{equation}\label{solutio}
(n-n^r_k, q-q^r_k)\in C^0([0, +\infty); L^2),\ (n, q)_x\in C^0([0, +\infty); L^2),\ \fix{n_{xx}\in L^2((0,+\infty);L^2)},
\end{equation}
and
\begin{equation}\label{approximate}
\lim\limits_{t\rightarrow \infty}\sup\limits_{x\in \mathbb{R}}
\left|(n,q)(x,t)-(n^r_k,q^r_k)(x,t)\right|=0,
\end{equation}
where $(N_{k0},Q_{k0})(x)$ is the initial value of the \fix{$\varepsilon$-dependent } smooth approximate rarefaction wave $(N_k,Q_k)(x,t)$ constructed in \eqref{U}.
\end{theorem}

We proceed to investigate the stability of the composite wave of two rarefaction waves.
\fix{
Following the Riemann-problem analysis in \cite{F.He 24,T.Li 22}, we denote by $RR(n_-,q_-)$ the admissible two-rarefaction region (within $n>0$) for which the Riemann solution is uniquely represented by $1$-rarefaction followed by $2$-rarefaction. Throughout this paper, the notation $RR(n_-,q_-)$ is used only in this sense. For $(n_+,q_+)\in RR(n_-,q_-)$, 
} the Riemann problem \eqref{conservation law}-\eqref{n0,q0} has a unique solution denoted by $(n^r,q^r)(x,t)$ which
can be constructed as follows.

One can find a unique state $(\bar{n},\bar{q})$ on the $1$-rarefaction wave curve $R_1(n_-,q_-)$, i.e., $(\bar{n},\bar{q})\in R_1(n_-,q_-)$, such that $(n_+,q_+)$ is on the
$2$-rarefaction wave curve $R_2(\bar{n},\bar{q})$.
Let $(n_1^r, q_1^r)(x,t)$ denote the $1$-rarefaction wave connecting $(n_-,q_-)$ to $(\bar{n},\bar{q})$ and $(n_2^r, q_2^r)(x,t)$ denote the $2$-rarefaction wave connecting $(\bar{n},\bar{q})$ to $(n_+,q_+)$. Then the composite rarefaction wave $(n^r, q^r)$ $(x,t)$ is a linear superposition of $(n_1^r, q_1^r)(x,t)$ and $(n_2^r, q_2^r)(x,t)$:
\begin{equation}\label{n1+n2}
(n^r, q^r)(x,t):=(n_1^r, q_1^r)(x,t)+(n_2^r, q_2^r)(x,t)-(\bar{n},\bar{q}).
\end{equation}

Our second result is to show that \fix{when the initial data $(n_0,q_0)$ are sufficiently close in $H^1$ to the chosen smooth composite initial profile $(N_0,Q_0)$}, the solution of the system \eqref{1.1}-\eqref{i d} will tend to the composite rarefaction wave $(n^r,q^r)(x,t)$ as $t\rightarrow +\infty$.

\begin{theorem}\label{the1.2}
\fix{Fix $\theta>\frac32$ and admissible end states $(n_-,q_-)$ and $(n_+,q_+)\in RR(n_-,q_-)$ with $n_\pm>0$, there exist $\varepsilon_0,\delta_1>0$ such that, for every $0<\varepsilon\leq\varepsilon_0$, the composite rarefaction wave $(n^r,q^r)(x,t)$ constructed in \eqref{n1+n2} is nonlinearly stable in the following sense: if }
\begin{equation}
\|n_0-N_0, q_0-Q_0\|_{H^1}\leq\delta_1,
\end{equation}
then the initial value problem \eqref{1.1}-\eqref{i d} has a unique global solution $(n,q)(x,t)$ satisfying
\begin{equation}\label{1.9}
(n-n^r, q-q^r)\in C^0([0, +\infty); L^2),\
(n, q)_x\in C^0([0, +\infty); L^2),\
\fix{n_{xx}\in L^2((0,+\infty);L^2)},
\end{equation}
and
\begin{equation}\label{stable}
\lim\limits_{t\rightarrow \infty}\sup\limits_{x\in \mathbb{R}}
\left|(n,q)(x,t)-(n^r,q^r)(x,t)\right|=0.
\end{equation}
\fix{Here $(N_0,Q_0)(x)$ is the initial value of the spatially separated smooth composite approximation $(N,Q)(x,t)$ constructed in Section \ref{Sect.2}; its two elementary Burgers profiles are shifted by $L_\varepsilon=\varepsilon^{-2}$ in opposite directions.}
\end{theorem}

\begin{remark}
In biology our results imply that 
the bacterial distribution will become increasingly sparse in the form of the rarefaction wave if initially it is close to $(N_0,Q_0)$.
\end{remark}

\fix{
\begin{remark}
No smallness assumption is imposed on the rarefaction-wave strength for each fixed admissible pair of end states in the region $n>0$. Thus the wave strength need not be small. This assertion is not uniform over unbounded families of end states, and the initial $H^1$ perturbation must still be small.
\end{remark}}

\fix{Theorem \ref{the1.2} differs from Xin's two-mode stability result \cite{Z.P Xin 89} by eliminating the small-amplitude assumption for the fixed composite wave in this chemotaxis model. Since the viscosity matrix of \eqref{1.1} is degenerate,  the positive definite viscosity theory of Xin \cite{Z.P Xin 89} or Yang and Zhao \cite{T. Yang 05} cannot be invoked directly.}
We prove the first theorem by observing two key ingredients: One is 
the monotonicity of the approximations of waves in  both $x$ and $t$, which provides a \lq\lq good\rq\rq \ sign in the proof of  basic energy estimate (see Lemma \ref{L2 estimate}). The other is the management of $(q-Q_k)_x$, i.e. $\psi_x$ in the perturbation equations. Since $\psi$ satisfies a first-order equation (see \eqref{rewrite}), it is challenging to estimate $\psi_x$ directly. Instead, thanks to the specific coupling structure of the perturbation equations, we replace $\psi_x$ in terms of $\phi$ in the $L^2$ estimate (see \eqref{psix}) and $H^1$ estimate (see \eqref{psitx}). Then, by exploring the strong dissipation of the parabolic equation of $\phi$, we successfully close the \emph{a priori} estimate. \fix{In the two-wave case, a naive superposition of two strong smoothed rarefaction waves does not yield a sufficiently small time-integrated interaction error term. We therefore separate their initial Burgers centers by $2L_\varepsilon$, where $L_\varepsilon=\varepsilon^{-2}$, and prove
\[
\int_0^\infty\|g(N,Q)_x(t)\|_{L^2}\,dt
\leq C_\theta\varepsilon^{1/2},\qquad
\int_0^\infty\|g(N,Q)_x(t)\|_{L^2}^2\,dt
\leq C_\theta\varepsilon^2,
\]
see Lemma \ref{properties of N',Q'}. Since the shifts are fixed for each $\varepsilon>0$,  they do not change the asymptotic rarefaction pattern. 
}


The rest of this paper is organized as follows. In Section \ref{Sect.2}, we construct smooth approximations of the single  rarefaction waves and the composite  rarefaction waves of the inviscid system \eqref{conservation law} by following the framework established by Matsumura and Nishihara \cite{Matsumura 92}. Moreover, we present some basic properties, such as monotonicity, of the approximations. In Section \ref{Sect.3}, we prove the asymptotic stability of single rarefaction waves. Section \ref{Sect.4} is devoted to the proof of stability of composite rarefaction waves.

\section{Rarefaction waves}\label{Sect.2}

This section is devoted to constructing a smooth approximation $(N,Q)(x,t)$ of the rarefaction wave $(n^r,q^r)(x,t)$ and presenting some preliminary estimates for $(N,Q)(x,t)$. \fix{In what follows,  parameter $0<\varepsilon\leq1$ and $\theta>\frac32$ is fixed. Unless otherwise indicated, $C$ denotes a generic positive constant which may change from line to line and may depend on $\theta$ and on the fixed admissible end states, but is independent of $t$, the continuation time $T$ and $\varepsilon$. When two waves are present, their intermediate state is assumed to satisfy $\bar{n}>0$.
}

\subsection{single-mode case}
We first construct a smooth approximation of a single rarefaction wave.
Suppose $(n_+,q_+)\in R_k(n_-,q_-)$, $k=1,2$, and consider the Riemann problem for the inviscid Burgers equation:
\begin{equation}\label{Burgers}
\left\{\begin{aligned}
&\frac{\partial{w^r_k}}{\partial_t}
+w^r_k\frac{\partial{w^r_k}}{\partial_x}=0,\\
&w^r_k(x,0)=w^r_{k0}(x),
\end{aligned}
\right.
\end{equation}
where the initial value $w^r_{k0}(x)$ satisfies
\begin{equation}\label{wr0}
w^r_{k0}(x)=
\left\{\begin{aligned}
&\lambda_k(n_-,q_-),\ \ \ \ \ x<0,\\
&\lambda_k(n_+,q_+),\ \ \ \ \ x>0.
\end{aligned}
\right.
\end{equation}
Here $\lambda_k$ $(k=1,2)$ is given by \eqref{1.8}. It is well-known (cf. \cite{J. A. 83}) that \eqref{Burgers}-\eqref{wr0} has a continuous weak solution
$w^r_k(x,t)$ in the form of
\begin{equation}\label{wr}
w^r_k(x,t)=
\left\{\begin{aligned}
&\lambda_k(n_-,q_-),\ \ \ \ \ \frac{x}{t}\leq\lambda_k(n_-,q_-),\\
&\frac{x}{t},\ \ \ \ \ \ \ \ \ \ \ \ \ \ \ \ \lambda_k(n_-,q_-)\leq\frac{x}{t}\leq \lambda_k(n_+,q_+),\\
&\lambda_k(n_+,q_+),\ \ \ \ \ \frac{x}{t}\geq\lambda_k(n_+,q_+).
\end{aligned}
\right.
\end{equation}
Set \begin{equation}\label{nr,qr}
\begin{aligned}
\lambda_k((n^r_k,q^r_k)(x,t))\triangleq w^r_k(x,t), \ h_k((n^r_k,q^r_k)(x,t))\triangleq h_k(n_-,q_-),
\end{aligned}
\end{equation} where $h_k$ $(k=1,2)$ is the $k$-th Riemann invariant given by \eqref{RI}.
Since $(n_+,q_+)\in R_k(n_-,q_-)$, it holds that $\lambda_k((n^r_k,q^r_k)(x,t))
\geq \lambda_k(n_-,q_-)$, which implies $$(n^r_k,q^r_k)(x,t) \in R_k(n_-,q_-).$$
Hence, the unique solution $(n^r_k,q^r_k)(x,t)$ of the Riemann problem \eqref{conservation law}-\eqref{n0,q0} is given by \eqref{nr,qr}.

We approximate $w_k^r(x,t)$ by smooth functions $w_k(x,t)$ that are the solutions of the following initial value problem
\begin{equation}\label{smooth}
\left\{\begin{aligned}
&\frac{\partial{w_k}}{\partial_t}
+w_k\frac{\partial{w_k}}{\partial_x}=0,\\
&w_k(x,0)=w_{k0}(x)
\end{aligned}
\right.
\end{equation}
with
\begin{equation}\label{w0}
\begin{aligned}
w_{k0}(x):=\frac{\lambda_k(n_+,q_+)+\lambda_k(n_-,q_-)}{2}
+\frac{\lambda_k(n_+,q_+)-\lambda_k(n_-,q_-)}{2}
\kappa_\theta\int^{\varepsilon x}_0(1+y^2)^{-\theta}dy,
\end{aligned}
\end{equation}
where $\varepsilon>0$ is a small  parameter and $\kappa_\theta$ is the constant such that $\kappa_\theta\int^\infty_0(1+y^2)^{-\theta}dy=1$ for each $\theta>\frac{3}{2}$. As in \cite[Lemma 2.1]{Matsumura 92}, one can see that
$w_k(x,t)$ {\color{red}satisfies} the following estimates.

\begin{lemma}[cf. \cite{Matsumura 92}-Lemma 2.1]\label{lem2.1ref} Fix $k=1,2$. If $\lambda_k(n_-,q_-)<\lambda_k(n_+,q_+)$, then the problem \eqref{smooth} has a unique global smooth solution $w_k(x,t)$ satisfying the following:

\begin{enumerate}

  \item[(1).] $\lambda_k(n_-,q_-)< w_k(x,t)<\lambda_k(n_+,q_+)$, $w_{kx}(x,t)>0$,
$\forall$ $(x,t)\in \mathbb{R}\times\mathbb{R}_+$.

\item[(2).] For any $p\in{\color{red}[1,+\infty]}$, there exists a constant $C_{p,\theta}>0$ such that
 \begin{equation}\label{wkx}
 \begin{aligned}
&\left\|{\color{red}w_{kx}(\cdot, t)}\right\|_{L^p}
 \leq C_{p,\theta}
\min\{\varepsilon^{1-1/p},
t^{-1+1/p}\}, \forall\ t> 0,\\
&\left\|w_{kxx}(\cdot, t)\right\|_{L^p}
\leq C_{p,\theta}
\min\{\varepsilon^{2-1/p},
\varepsilon^{(1-\frac{1}{2\theta})(1-\frac{1}{p})}
t^{-1-\frac{p-1}{2p\theta}}\}, \forall \ t>0.
\end{aligned}
\end{equation}

\item[(3).] If $\lambda_k(n_-,q_-)>0$, then for  $x\leq0$, $t\in \mathbb{R}_+$, it holds that
\begin{equation}\label{wk-w-}
 \begin{aligned}
\left|w_k(x,t)-\lambda_k(n_-,q_-)\right|
 &\leq C_\theta
(1+(\varepsilon x)^2)^{-\theta/3}
[1+(\varepsilon \lambda_k(n_-,q_-)t)^2]^{-\theta/3},\\ \left|w_{kx}(x, t)\right|
&\leq C_\theta
\varepsilon(1+(\varepsilon x)^2)^{-\theta/2}
[1+(\varepsilon \lambda_k(n_-,q_-)t)^2]^{-\theta/2}.
\end{aligned}
\end{equation}

\item[(4). ] If $\lambda_k(n_+,q_+)<0$, then for ${\color{red}x\geq0}$, $t\in \mathbb{R}_+$, it holds that
\begin{equation}\label{wk-w+}
 \begin{aligned}
\left|w_k(x,t)-\lambda_k(n_+,q_+)\right|
 &\leq C_\theta
(1+(\varepsilon x)^2)^{-\theta/3}
[1+(\varepsilon \lambda_k(n_+,q_+)t)^2]^{-\theta/3},\\ \left|w_{kx}(x, t)\right|
&\leq C_\theta
\varepsilon(1+(\varepsilon x)^2)^{-\theta/2}
[1+(\varepsilon \lambda_k(n_+,q_+)t)^2]^{-\theta/2}.
\end{aligned}
\end{equation}

\item[(5).] $\lim\limits_{t\rightarrow \infty}
\sup\limits_{x\in \mathbb{R}}
\left|w_k(x,t)-w_k^r(x,t)\right|=0.$

\end{enumerate}

\end{lemma}

By the characteristic method, one can find that the solution of \eqref{smooth} is expressed by the form
\begin{equation}\label{w}
\begin{aligned}
w_k(x,t)=w_{k0}(x_0(x,t)),
\end{aligned}
\end{equation}
where $x_0(x,t)$ is given by the relation
\begin{equation}\label{x}
\begin{aligned}
x=x_0(x,t)+w_{k0}(x_0(x,t))t.
\end{aligned}
\end{equation}
By \eqref{w} and \eqref{x}, one can easily show Lemma \ref{lem2.1ref}. We refer \cite{Matsumura 92} for the detail of the proof.

Now we define $(N_k,Q_k)(x,t)$ $(k=1,2)$ by
\begin{equation}\label{U}
\begin{aligned}
(N_k,Q_k)(x,t)\in R_k(n_-,q_-),\quad  \lambda_k((N_k,Q_k)(x,t))=w_k(x,t), \ k=1,2.
\end{aligned}
\end{equation}
The next lemma shows that $(N_k,Q_k)(x,t)$ is the desired smooth approximation of $(n^r_k,q^r_k)(x,t)$.

\begin{lemma}\label{lem2.1}
$(N_k,Q_k)(x,t)$ $(k=1,2)$ is an approximation of $(n^r_k,q^r_k)(x,t)$  $(k=1,2)$ in the following sense:

\begin{enumerate}

\item[(1).] $(N_k,Q_k)(x,t)$ satisfies the system \eqref{equivalent system}.

\item[(2).] $\lim\limits_{t\rightarrow \infty}
\sup\limits_{x\in \mathbb{R}}
\left|(N_k,Q_k)(x,t)-(n^r_k,q^r_k)(x,t)\right|=0.$

\end{enumerate}

\end{lemma}

\begin{proof}
By \eqref{genuinely nonlinear} and \eqref{U}, we have
\begin{equation*}
\left(\begin{aligned}
&N_k\\&Q_k
\end{aligned}
\right)_t=\frac{w_{kt}}{\nabla \lambda_k\cdot r_k}
 r_k,\quad
\left(\begin{aligned}
&N_k\\&Q_k
\end{aligned}
\right)_x=\frac{w_{kx}}{\nabla \lambda_k\cdot r_k}
 r_k,\ \ \ \ k=1,2.
\end{equation*}
Thus,
\begin{equation}\label{2.13}
\begin{aligned}
\left(\begin{aligned}
&N_k\\&Q_k
\end{aligned}
\right)_t
+\left(\begin{aligned}
&-Q_k\ \ -N_k\\&-1\ \ \ \ \ \ \ 0
\end{aligned}\right)
\left(\begin{aligned}
&N_k\\&Q_k
\end{aligned}
\right)_x
&=\frac{1}{\nabla \lambda_k\cdot r_k}
(w_{kt}+\lambda_kw_{kx}) r_k\\
&=\frac{1}{\nabla \lambda_k\cdot r_k}
(w_{kt}+w_k(x,t)w_{kx}) r_k=0,
\end{aligned}
\end{equation}
which shows (1).

The conclusion (2) follows from \eqref{Rk}, \eqref{nr,qr}, \eqref{U}, the fact that $\nabla \lambda_k$ and $\nabla h_k$ are linearly independent and Lemma \ref{lem2.1ref}-(5).
\end{proof}

\begin{lemma}\label{properties of N,Q}
The smooth functions $(N_k,Q_k)(x,t)$, $k=1,2$, constructed in \eqref{U} have the following properties.
\begin{enumerate}
\item[(1).] \fix{For both $k=1,2$, $Q_{kx}<0$ and $N_{kt}<0$. Moreover,
\[
N_{1x}<0,\quad Q_{1t}<0,\quad N_{2x}>0,\quad Q_{2t}>0,\quad \forall \ x\in \mathbb{R},\ t>0.
\]}

\item[(2).] For any $p\in\fix{[1,+\infty]}$, there exists $ \ C_p>0$ such that
\begin{equation}\label{Ux}
\left\|\frac{\partial N_k}{\partial x}\right\|_{L^p}, \
\left\|\frac{\partial Q_k}{\partial x}\right\|_{L^p} \leq C_p
\min\{\varepsilon^{1-1/p},\
t^{-1+1/p}\},\forall \ t> 0.
\end{equation}
\item[(3).] For any $p\in\fix{(1,+\infty]}$, there exists $\ C_{p\theta}>0$ such that
\begin{equation}\label{Ul}
\|N_{kxx}\|_{L^p}+\|Q_{kxx}\|_{L^p}
\leq C_{p,\theta}\min\left\{\varepsilon^{2-1/p},
\fix{\varepsilon^{(1-\frac1{2\theta})(1-\frac1p)}
t^{-1-\frac{p-1}{2p\theta}}}\right\},\qquad t>0.
\end{equation}
For $p\in(1,+\infty]$, there exists $\ C_p>0$ such that
\begin{equation}\label{Nxx}
\int_0^t\big(\|N_{kxx}(\tau)\|_{L^p}+\|Q_{kxx}(\tau)\|_{L^p}\big)d\tau
\leq C_{p}\varepsilon^{1-1/p},\qquad t>0.
\end{equation}
\begin{fixblock}
At the endpoint $p=1$, we use the squared time-integrability:
\begin{equation}\label{single-curvature-square}
\int_0^\infty\big(\|N_{kxx}\|_1^2+\|Q_{kxx}\|_1^2\big)dt\le C
\varepsilon,
\quad \int_0^\infty\big(\|N_{kxx}\|_2^2+\|Q_{kxx}\|_2^2\big)dt\le C
\varepsilon^2.
\end{equation}
\end{fixblock}
\item[(4).] There exists $C>0$ such that
\begin{equation}\label{single-struct}
|N_{kx}|+|N_{kt}|+|Q_{kt}|\leq C|Q_{kx}|
\qquad\text{on }\mathbb R\times[0,\infty).
\end{equation}
\end{enumerate}
\end{lemma}

\begin{proof}
{\color{red}
We prove the assertions for both characteristic families simultaneously. Fix $k\in\{1,2\}$,} from \eqref{Rk} and \eqref{U}, we get
\begin{equation*}
\begin{aligned}
\nabla \lambda_{\color{red}{k}}\cdot (N_{\color{red}{kx}}\ Q_{\color{red}{kx}})^T
=w_{\color{red}{kx}},\ \
\nabla h_{\color{red}{k}}\cdot (N_{\color{red}{kx}}\ Q_{\color{red}{kx}})^T=0.
\end{aligned}
\end{equation*}
Hence,
\begin{equation}\label{'NxQx}
\begin{aligned}
N_{\color{red}{kx}}=
-\frac{h_{{\color{red}{k}}q}}
{h_{{\color{red}{k}}n}}Q_{{\color{red}{k}}x},\ \
(\lambda_{{\color{red}{k}}q}-\lambda_{{\color{red}{k}}n}
\frac{h_{{\color{red}{k}}q}}{h_{{\color{red}{k}}n}})Q_{{\color{red}{k}}x}
=w_{{\color{red}{k}}x},
\end{aligned}
\end{equation}
and noting that \begin{equation*}
\frac{h_{{\color{red}{k}}q}}
{h_{{\color{red}{k}}n}}=
-\lambda{\color{red}_k}.
\end{equation*}
Similarly, we take the partial derivative of \eqref{Rk} and \eqref{U} with respect to $t$, which implies
\begin{equation*}
\begin{aligned}
\nabla \lambda_{\color{red}{k}}\cdot (N_{{\color{red}{k}}t}\ Q_{{\color{red}{k}}t})^T=w_{{\color{red}{k}}t},\
\nabla h_{\color{red}{k}}\cdot (N_{{\color{red}{k}}t}\ Q_{{\color{red}{k}}t})^T=0.
\end{aligned}
\end{equation*}
A direct calculation yields
\begin{equation}
\begin{aligned}\label{'NtQt}
N_{{\color{red}{k}}t}
=-\frac{h_{{\color{red}{k}}q}}{h_{{\color{red}{k}}n}}
Q_{{\color{red}{k}}t},
\end{aligned}
\end{equation}which, along with Lemma \ref{lem2.1}-(1), leads to the following identities
{\color{red}\begin{equation}\label{NxQx}
\begin{aligned}
N_{kx}=-\lambda_k Q_{kx},\ \
Q_{kt}=N_{kx},\ \
N_{kt}=\lambda_k^2Q_{kx}.
\end{aligned}
\end{equation}
It hence follows from Lemma \ref{lem2.1ref}-(1) that $Q_{kx}<0$ and $N_{kt}<0$. Moreover, $\lambda_1<0<\lambda_2$ in $n>0$, which yields $N_{1x}<0$, $Q_{1t}<0$, $N_{2x}>0$ and $Q_{2t}>0$.}
This completes the proof of (1).

{\color{red}
It follows from \eqref{'NxQx} that
\begin{equation*}
\begin{aligned}
N_{kx}&=
f_k^1(N_k,Q_k)w_{kx},
&Q_{kx}&=
g_k^1(N_k,Q_k)w_{kx},\\
N_{kxx}&=
f_k^2(N_k,Q_k)w_{kx}^2
+f_k^3(N_k,Q_k)w_{kxx},
&Q_{kxx}&=
g_k^2(N_k,Q_k)w_{kx}^2
+g_k^3(N_k,Q_k)w_{kxx},
\end{aligned}
\end{equation*}
where $f_k^i(N_k,Q_k)$ and $g_k^i(N_k,Q_k),\ i=1.2.3,$ are smooth functions. Then there is a constant $C>0$, independent of $\varepsilon$ and $t$, such that
\begin{equation}\label{NtQt}
|N_{kx}|,|Q_{kx}|\leq C|w_{kx}|,\ \
|N_{kxx}|,|Q_{kxx}|\leq C\big(|w_{kxx}|+|w_{kx}|^2\big).
\end{equation}
The first inequality and Lemma \ref{lem2.1ref}-(2) immediately give \eqref{Ux} for every $1\leq p\leq\infty$, proving (2).

To prove (3), we estimate the quadratic term in \eqref{NtQt}. By H\"older's inequality and Lemma \ref{lem2.1ref}-(2), for $1\leq p\leq\infty$,
\[
\|w_{kx}^2(t)\|_{L^p}
\leq\|w_{kx}(t)\|_{L^\infty}\|w_{kx}(t)\|_{L^p}
\leq C_{p,\theta}\min\{\varepsilon^{2-1/p},\,t^{-2+1/p}\}.
\]
Let
\[
\alpha_p=\Big(1-\frac{1}{2\theta}\Big)\Big(1-\frac1p\Big),\qquad
\beta_p=\frac{1-1/p}{2\theta}.
\]
For $t\geq\varepsilon^{-1}$, the identity $2-1/p=1+\beta_p+\alpha_p$ implies
$t^{-2+1/p}\leq\varepsilon^{\alpha_p}t^{-1-\beta_p}.
$ 
For $0<t\leq\varepsilon^{-1}$ the corresponding bound follows from $\varepsilon^{2-1/p}\leq\varepsilon^{\alpha_p}t^{-1-\beta_p}$. Thus the preceding quadratic estimate obeys the same time-decay bound as $\|w_{kxx}(t)\|_{L^p}$ in \eqref{wkx}. Inserting these estimates into \eqref{NtQt} proves \eqref{Ul} whenever $1<p\leq\infty$. For such $p$, $\beta_p>0$, and splitting the time integral at $t_0=\varepsilon^{-1}$ gives
\[
\begin{aligned}
\int_0^\infty\!\big(\|N_{kxx}(t)\|_{L^p}+\|Q_{kxx}(t)\|_{L^p}\big)\,dt
&\leq C_{p,\theta}\left(
\varepsilon^{2-1/p}\varepsilon^{-1}
+\varepsilon^{\alpha_p}\int_{\varepsilon^{-1}}^\infty\! t^{-1-\beta_p}\,dt\right)\\
&\leq C_{p,\theta}\varepsilon^{1-1/p},
\end{aligned}
\]
which is \eqref{Nxx}.

When $p=1$, the preceding time integral need not converge, so we use the squared estimate in \eqref{single-curvature-square}.
It follows from \eqref{w} and \eqref{x} that
\[
w_{kx}(x,t)=\frac{w_{k0}'(\xi)}{1+t w_{k0}'(\xi)},
\qquad x=\xi+t w_{k0}(\xi).
\]
The function $w_{k0}'(\xi)$ is positive, even, and strictly decreasing in $|\xi|$. Since $\xi\mapsto x$ is strictly increasing 
, the function $w_{kx}(\cdot,t)$ is unimodal and tends to zero at both spatial infinities. Consequently,
\[
\|w_{kxx}(t)\|_{L^1}=2\|w_{kx}(t)\|_{L^\infty}
\leq C\min\{\varepsilon,t^{-1}\}.
\]
Furthermore, monotonicity of $w_k$ gives
$\|w_{kx}(t)\|_{L^1}=\lambda_k(n_+,q_+)-\lambda_k(n_-,q_-)$, whence
\[
\|w_{kx}^2(t)\|_{L^1}
\leq\|w_{kx}(t)\|_{L^\infty}\|w_{kx}(t)\|_{L^1}
\leq C\min\{\varepsilon,t^{-1}\}.
\]
Combining this with \eqref{NtQt} and the preceding $L^1$ bound, we find
\[
\|N_{kxx}(t)\|_{L^1}+\|Q_{kxx}(t)\|_{L^1}
\leq C\min\{\varepsilon,t^{-1}\}.
\]
Thus
\[
\int_0^\infty\!\big(\|N_{kxx}(t)\|_{L^1}^2+\|Q_{kxx}(t)\|_{L^1}^2\big)\,dt
\leq C\left(\varepsilon^2\int_0^{\varepsilon^{-1}}dt
+\int_{\varepsilon^{-1}}^\infty t^{-2}dt\right)
\leq C\varepsilon.
\]
Similarly, \eqref{Ul} with $p=2$ yields
\[
\begin{aligned}
\int_0^\infty\!\big(\|N_{kxx}(t)\|_{L^2}^2+\|Q_{kxx}(t)\|_{L^2}^2\big)\,dt
&\leq C\left(\varepsilon^3\int_0^{\varepsilon^{-1}}dt
+\varepsilon^{1-\frac{1}{2\theta}}
\int_{\varepsilon^{-1}}^\infty t^{-2-\frac{1}{2\theta}}dt\right)\\
&\leq C\varepsilon^2.
\end{aligned}
\]
This proves \eqref{single-curvature-square} and completes (3).

Finally, \eqref{NxQx} gives
\[
|N_{kx}|+|N_{kt}|+|Q_{kt}|
=\big(2|\lambda_k|+\lambda_k^2\big)|Q_{kx}|
\leq C|Q_{kx}|,
\]
where $\lambda_k$ is bounded along the prescribed rarefaction arc. Hence \eqref{single-struct} holds and (4) follows.
}
\end{proof}

\subsection{Two-mode case}

We proceed to construct a smooth approximation of the composite rarefaction wave. \fix{The two elementary smooth waves are given opposite spatial phase shifts. This preserves all single-wave estimates and
makes their interaction small already at the initial time.

Let $\widehat W_i(x,t)$, $i=1,2$, solve the following initial value problems for the inviscid Burgers equation,
\begin{equation}\label{smooth'}
\widehat W_{it}+\widehat W_i\widehat W_{ix}=0,\qquad
\widehat W_i(x,0)=\widehat W_{i0}(x),
\end{equation}
with
\begin{equation}\label{wi0}
\begin{aligned}
\widehat W_{10}(x)&=\frac{\lambda_1(\bar n,\bar q)+\lambda_1(n_-,q_-)}2
+\frac{\lambda_1(\bar n,\bar q)-\lambda_1(n_-,q_-)}2
\kappa_\theta\int_0^{\varepsilon x}(1+y^2)^{-\theta}\,dy,\\
\widehat W_{20}(x)&=\frac{\lambda_2(n_+,q_+)+\lambda_2(\bar n,\bar q)}2
+\frac{\lambda_2(n_+,q_+)-\lambda_2(\bar n,\bar q)}2
\kappa_\theta\int_0^{\varepsilon x}(1+y^2)^{-\theta}\,dy,
\end{aligned}
\end{equation}
where $\varepsilon>0$ is a small parameter and $\kappa_\theta$ is the constant such that $\kappa_\theta\int^\infty_0(1+y^2)^{-\theta}dy=1$ for $\theta>\frac{3}{2}$.

Set
\begin{equation}\label{phase-shift}
L_\varepsilon:=\varepsilon^{-2},\qquad
W_1(x,t):=\widehat W_1(x+L_\varepsilon,t),\qquad
W_2(x,t):=\widehat W_2(x-L_\varepsilon,t).
\end{equation}
By translation invariance, $W_i$ are again smooth Burgers solutions with the same end states, and all
translation-invariant estimates of Lemma \ref{lem2.1ref} remain valid.} Then for the composite rarefaction wave $(n^r, q^r)(x,t)$ constructed in \eqref{n1+n2}, its smooth approximation $(N,Q)(x,t)$ can be defined by
\begin{equation}\label{N,Q}
\begin{aligned}
(N,Q)(x,t)=(N_1,Q_1)(x,t)+(N_2,Q_2)(x,t)
-(\bar{n},\bar{q}).
\end{aligned}
\end{equation}
Here $(N_i,Q_i)(x,t)$ $(i=1,2)$ are defined by the following relations
\begin{equation}\label{Ni,Qi}
\begin{aligned}
&(N_1,Q_1)(x,t)\in R_1(n_-,q_-),\quad & \lambda_1((N_1,Q_1)(x,t))=W_1(x,t);\\
&(N_2,Q_2)(x,t)\in R_2(\bar{n},\bar{q}),\quad & \lambda_2((N_2,Q_2)(x,t))=W_2(x,t),
\end{aligned}
\end{equation}

As in the single-mode case, by using the implicit function theorem and the characteristic method, one can easily show that \eqref{Ni,Qi} gives smooth functions
\begin{equation}\label{Ni}
\begin{aligned}
(N_i,Q_i)(x,t)=(N_i,Q_i)(x_{i 0}(x,t)),\ \ \ \ i=1,2,
\end{aligned}
\end{equation}
where
\begin{equation}\label{x^i}
\begin{aligned}
x=x_{i 0}(x,t)+W_{i 0}(x_{i 0}(x,t))t,\ \ \ \ i=1,2.
\end{aligned}
\end{equation}
\fix{
The phase shifts do not change the asymptotic Riemann wave. If $W_i^r$ denotes the corresponding
Burgers rarefaction wave and $\sigma_1=1$, $\sigma_2=-1$, then
\[
\sup_x|W_i(x,t)-W_i^r(x,t)|
\leq \sup_y|\widehat W_i(y,t)-W_i^r(y,t)|
+\sup_x|W_i^r(x+\sigma_iL_\varepsilon,t)-W_i^r(x,t)|.
\]
The first term tends to zero by Lemma \ref{lem2.1ref}-(5), and the second is bounded by
$L_\varepsilon/t$. It follows that
\begin{equation}\label{limNi}
\lim_{t\to\infty}\sup_x|(N_i,Q_i)(x,t)-(n_i^r,q_i^r)(x,t)|=0,\qquad i=1,2.
\end{equation}}
Thus, the smooth approximation $(N,Q)$ satisfies
\begin{equation}\label{N'Q'}
\left\{
\begin{aligned}
&N_{t} -(NQ)_{x}=-g(N,Q)_x,\\
&Q_{t}-N_{x}=0, 
\end{aligned}
\right.
\end{equation}
and
\begin{equation}\label{limN}
\lim\limits_{t\rightarrow \infty}
\sup\limits_{x\in \mathbb{R}}
\left|(N,Q)(x,t)-(n^r,q^r)(x,t)\right|=0,
\end{equation}
where $g(N,Q)=NQ-N_1Q_1-N_2Q_2$.

Analogous to Lemma \ref{properties of N,Q}, we have the following estimates for $(N,Q)$.

\begin{lemma}\label{properties of N',Q'}
\fix{
The spatially separated composite approximation $(N,Q)$ satisfies:
\begin{enumerate}
\item[(1).] $Q_x<0$, $N_t<0$, \fix{$Q_t=N_x$}, and $N(x,t)\geq\bar n>0$. Moreover,
\begin{equation}\label{composite-struct}
|N_x|+|N_t|+|Q_t|\leq C|Q_x|\fix{,\qquad |N_x|^2\le |N_t|\,|Q_x|\le C\varepsilon|N_t|}.
\end{equation}
\item[(2).] For every $p\in[1,+\infty]$, there exists $ \ C_p>0$ such that
\begin{equation}\label{Ux'}
\|N_x\|_{L^p}+\|Q_x\|_{L^p}
\leq C_p\min\{\varepsilon^{1-1/p},t^{-1+1/p}\}.
\end{equation}
\item[(3).] For every $p\in[1,+\infty]$, there exists $ \ C_{p\theta}>0$ such that
\begin{equation}\label{Ul'}
\|N_{xx}\|_{L^p}+\|Q_{xx}\|_{L^p}
\leq C_{p,\theta}\min\left\{\varepsilon^{2-1/p},
\varepsilon^{(1-\frac1{2\theta})(1-\frac1p)}
t^{-1-\frac{p-1}{2p\theta}}\right\}.
\end{equation}
For $p\in(1,+\infty]$, there exists $ \ C_{p}>0$ such that
\begin{equation}\label{N'xx}
\int_0^t(\|N_{xx}(\tau)\|_{L^p}+\|Q_{xx}(\tau)\|_{L^p})d\tau
\leq C_{p}\varepsilon^{1-1/p}.
\end{equation}
Moreover, the additional bounds needed for the gradient energies are
\begin{equation}\label{composite-curvature-square}
\int_0^\infty\!\big(\|N_{xx}(t)\|_1^2+\|Q_{xx}(t)\|_1^2\big)dt
\le C\varepsilon,\qquad
\int_0^\infty\!\big(\|N_{xx}(t)\|_2^2+\|Q_{xx}(t)\|_2^2\big)dt
\le C\varepsilon^2.
\end{equation}
\item[(4).] Let
$c_0:=\min\{-\lambda_1(\bar n,\bar q),\lambda_2(\bar n,\bar q)\}>0$. Then
\begin{equation}\label{gx}
\begin{aligned}
\|g(N,Q)_x(t)\|_{L^2}
\leq C_\theta\Big[
&\varepsilon^{2\theta/3}\min\{\varepsilon^{1/2},t^{-1/2}\}
(1+(c_0\varepsilon t)^2)^{-\theta/3}\\
&+\varepsilon^\theta(1+(c_0\varepsilon t)^2)^{-\theta/2}
\Big],
\end{aligned}
\end{equation}
and consequently
\begin{equation}\label{tgx}
\int_0^\infty\|g(N,Q)_x(t)\|_{L^2}\,dt\leq C\varepsilon^{1/2},
\qquad
\int_0^\infty\|g(N,Q)_x(t)\|_{L^2}^2\,dt\leq C\varepsilon^2.
\end{equation}
\end{enumerate}}
\end{lemma}
{\color{red}
\begin{proof}
Properties (1)-(3) follow from Lemma \ref{properties of N,Q} and translation invariance. Indeed $Q_{1x},Q_{2x}<0$ and $N_{1t},N_{2t}<0$. This directly yields from  $N_1\geq\bar n$ and $N_2\geq\bar n$ that $N=N_1+N_2-\bar n\geq\bar n$. Since $|Q_x|=|Q_{1x}|+|Q_{2x}|$, \eqref{single-struct} gives \eqref{composite-struct}.
A straightforward calculation yields
\begin{equation*}
  N_x=-\lambda_1Q_{1x}
  -\lambda_2Q_{2x},\ N_t=\lambda_1^2Q_{1x}
+\lambda_2^2Q_{2x}\  \text{and}\  Q_x=Q_{1x}+Q_{2x}.
\end{equation*}
By Cauchy–Schwarz inequality,  one has
$|N_x|^2\le(-N_t)(-Q_x)\le C\varepsilon|N_t|$. Noting that translation does not change $L^p$ norms, and adding the two single-wave
estimates \eqref{single-curvature-square}, we have
\eqref{composite-curvature-square}.

For the interaction, smooth dependence on $W_i$ along the fixed wave curves gives
\begin{equation}\label{gx'}
|g(N,Q)_x|
\leq C\Big(|W_{1x}||W_2-\lambda_2(\bar n,\bar q)|
+|W_{2x}||W_1-\lambda_1(\bar n,\bar q)|\Big).
\end{equation}
Let
\[
A_\varepsilon=(1+(\varepsilon L_\varepsilon)^2)^{-\theta/3},\qquad
B_\varepsilon=(1+(\varepsilon L_\varepsilon)^2)^{1/4-\theta/2}.
\]
For $x\leq0$, $x-L_\varepsilon\leq-L_\varepsilon$, Lemma
\ref{lem2.1ref}-(3) yields
\[
\|W_2-\lambda_2(\bar n,\bar q)\|_{L^\infty(\mathbb R_-)}
\leq C_\theta A_\varepsilon(1+(c_0\varepsilon t)^2)^{-\theta/3},
\]
and
\[
\|W_{2x}\|_{L^2(\mathbb R_-)}
\leq C_\theta\varepsilon^{1/2}B_\varepsilon
(1+(c_0\varepsilon t)^2)^{-\theta/2}.
\]
The latter follows from
\[
\int_{-\infty}^0(1+\varepsilon^2(x-L_\varepsilon)^2)^{-\theta}dx
\leq C\varepsilon^{-1}(1+(\varepsilon L_\varepsilon)^2)^{1/2-\theta}.
\]
On $\mathbb R_+$ the symmetric bounds follow from Lemma \ref{lem2.1ref}-(4) for
$W_1(x,t)=\widehat W_1(x+L_\varepsilon,t)$. Combining these tail bounds with
$\|W_{ix}(t)\|_{L^2}\leq C\min\{\varepsilon^{1/2},t^{-1/2}\}$ and splitting
$\mathbb R=\mathbb R_-\cup\mathbb R_+$ in \eqref{gx'} gives
\[
\|g(N,Q)_x(t)\|_{L^2}
\leq C_\theta\Big[
A_\varepsilon\min\{\varepsilon^{1/2},t^{-1/2}\}(1+(c_0\varepsilon t)^2)^{-\theta/3}
+\varepsilon^{1/2}B_\varepsilon(1+(c_0\varepsilon t)^2)^{-\theta/2}
\Big].
\]
The facts $L_\varepsilon=\varepsilon^{-2}$,
$A_\varepsilon\leq C\varepsilon^{2\theta/3}$ and
$B_\varepsilon\leq C\varepsilon^{\theta-1/2}$ lead to \eqref{gx}. Finally,
after the change of variables $s=\varepsilon t$, along with $\theta>\frac32$, one has
\[
\int_0^\infty\|g(N,Q)_x(t)\|_{L^2}dt
\leq C_\theta(\varepsilon^{2\theta/3-1/2}+\varepsilon^{\theta-1})
\leq C\varepsilon^{1/2},
\]
and
\[
\int_0^\infty\|g(N,Q)_x(t)\|_{L^2}^2dt
\leq C_\theta(\varepsilon^{4\theta/3}+\varepsilon^{2\theta-1})
\leq C\varepsilon^2,
\]
which proves \eqref{tgx}.
\end{proof}}

\section{Stability of single rarefaction waves}\label{Sect.3}

In this section, we investigate the stability of single rarefaction waves for the system \eqref{1.1}-\eqref{i d} and prove Theorem \ref{the1.1}. \fix{Throughout this section, we fix $\theta>\frac32$. The generic constant $C>0$ and the bootstrap smallness constants appearing below are independent of $T$ and $\varepsilon$. The energy estimates are first established for sufficiently smooth approximating solutions, and the general case follows by the standard
density and continuation arguments. For simplicity of notation,} we denote $\|\cdot\|_k:=\|\cdot\|_{H^k(\mathbb{R})}$ and $\|\cdot\|:=\|\cdot\|_{L^2(\mathbb{R})}$.

We decompose $(n,q)$ as
\begin{equation}\label{decompose}
(n,q)=(N_k+\phi,Q_k+\psi).
\end{equation}
\fix{
Set $\underline n_k:=\min\{n_-,n_+\}>0$.
Since the approximate rarefaction wave $N_k$ takes values between its two
end states, we have  \begin{equation}\label{Nk}
N_k(x,t)\ge \underline n_k>0,
\qquad (x,t)\in\mathbb{R}\times[0,\infty).
\end{equation}
}
Then by \eqref{1.1} and \eqref{2.13}, one can see that $(\phi,\psi)$ satisfies
\begin{equation}\label{rewrite}
\left\{
\begin{aligned}
&\phi_{t} -(\phi\psi+\phi Q_k+N_k\psi)_{x}=\phi_{xx}+N_{kxx},\\
&\psi_{t}-\phi_{x}=0,
\end{aligned}
\right.
\end{equation}
with initial value
\begin{equation}\label{phi0,psi0}
\begin{aligned}
(\phi_0,\psi_0)(x):=(\phi,\psi)(x,0)=\fix{(n_0-N_k(x,0),q_0-Q_k(x,0))}.
\end{aligned}
\end{equation}
We look for solutions of the system \eqref{rewrite} in the space
\begin{equation}\label{solution space}
\begin{aligned}
X(0,T):=\{(\phi,\psi)|(\phi,\psi)\in C^0([0,T];H^1),
\phi_x\in L^2((0,T);H^1),
\psi_x\in L^2((0,T);L^2)\fix{\}},
\end{aligned}
\end{equation}
for $T\in(0,+\infty]$. Set
\begin{equation}\label{assumption}
E(t):=\sup\limits_{0\leq \tau\leq t}\|(\phi,\psi)(\cdot,\tau)\|_1.
\end{equation}
By the Sobolev embedding theorem, we have
\begin{equation}\label{3.6}
\sup\limits_{0\leq \tau\leq t}(\|\phi(\cdot,\tau)\|_{L^\infty}+\|\psi(\cdot,\tau)\|_{L^\infty})\leq CE(t).
\end{equation}
We have the following global well-posedness of solutions to the system \eqref{rewrite}-\eqref{phi0,psi0}.

\begin{proposition}\label{phi stability}
\fix{Fix the admissible end states and $\theta>3/2$. There exist $\varepsilon_0,\delta_0>0$, independent of $0<\varepsilon\leq\varepsilon_0$, such that whenever $(\phi_0,\psi_0)\in H^1(\mathbb R)$ and $E(0)\leq\delta_0$, the system \eqref{rewrite}-\eqref{phi0,psi0} has a unique global solution $(\phi,\psi)\in X(0,\infty)$ satisfying}
\begin{equation}\label{priori estimates}
\begin{aligned}
\left\|(\phi,\psi)(\cdot,t)
\right\|_1^2
+&\int_{0}^{t}
\Big(\||Q_{kx}|^{1/2}\phi(\cdot,\tau) \|^{2}
+\||N_{kt}|^{1/2}\psi(\cdot,\tau) \|^{2}\\&+\left\|\phi_x(\cdot,\tau) \right\|_1^{2}+\left\|\psi_x(\cdot,\tau) \right\|^{2}\Big)d\tau\leq C_0 (\|(\phi_0,\psi_0)\|_1^2+\varepsilon),
\end{aligned}
\end{equation}
for any $t\in [0,\infty)$, where $\varepsilon>0$ is the small constant given in \eqref{w0}. Moreover, $(\phi,\psi)(x,t)$ has the following asymptotic behavior
	\begin{equation}\label{asymptotic}
		\sup _{x \in \mathbb{R}}|(\phi,\psi)(x, t)| \rightarrow 0 \text { as } t \rightarrow+\infty.
	\end{equation}	
\end{proposition}
The global existence of $(\phi,\psi)$ can be proved by the  local existence result and the \textit{a priori} estimate given below.

\begin{proposition}[Local existence]\label{local existence}
\fix{Fix the admissible end states and $\theta>\frac32$.}
For any $\Xi_0>0$, if the initial data $(\phi_0,\psi_0)$ satisfy
$\|(\phi_0,\psi_0)\|_1\leq \Xi_0$, then there exists a positive constant $T_0$ depending on $\Xi_0$ \fix{and the fixed parameters but uniform for $0<\varepsilon\leq1$}, such that the system \eqref{rewrite}-\eqref{phi0,psi0} has a unique solution $(\phi,\psi)$ in $X(0, T_0)$ satisfying
\begin{equation}\label{local estimates}
\begin{aligned}
\left\|(\phi,\psi)(\cdot,t)
\right\|_1^2
+&\int_{0}^{t}
\Big(\||Q_{kx}|^{1/2}\phi(\cdot,\tau) \|^{2}
+\||N_{kt}|^{1/2}\psi(\cdot,\tau) \|^{2}\\&+\left\|\phi_x(\cdot,\tau) \right\|_1^{2}+\left\|\psi_x(\cdot,\tau) \right\|^{2}\Big)d\tau\leq 4 \|(\phi_0,\psi_0)\|_1^2, \ \ \forall \ t \in [0, T_0].
\end{aligned}
\end{equation}
\end{proposition}

The local existence in Proposition \ref{local existence} can be shown by a standard iteration method and we omit the detail. To extend the local solution globally, it suffices to establish the following \emph{a priori} estimate.

\begin{proposition}[\emph{A priori} estimate]\label{A priori estimates}
Let $n_\pm>0$. Suppose that the system \eqref{rewrite}-\eqref{phi0,psi0} has a solution $(\phi,\psi)\in X(0,T)$ for some $T>0$.
Then there exist positive constants $\chi_0\leq 1$ and $C_0$ independent of $T$ such that if $$E(T)\leq \chi_0,$$
then the estimate \eqref{priori estimates} holds for any $t\in[0,T]$.

\end{proposition}

Before proving Proposition \ref{A priori estimates}, we first present an inequality of Gronwall's type.

\begin{lemma}\label{Gronwall}
Suppose $F(t)$ is a continuously differentiable function on $[0,T]$, $\beta(t)\geq0$ is integrable on $\mathbb{R}^+$. If $F(t)$ satisfies
\begin{equation}\label{2.45}
  \begin{aligned}
\left(F'(t)\right)^2
 \leq\beta^2(t) F(t) \text{ and } F(0)>0,\ t\in [0,T],
  \end{aligned}
  \end{equation}
then the following estimate holds:
\begin{equation}\label{Gronwall'}
  \begin{aligned}
F(t)
 \leq 2 F(0)+\frac{1}{2}\left(\int_{0}^{t}\beta(\tau)d\tau\right)^2 \text{ for all }t\in [0,T].
  \end{aligned}
  \end{equation}

\end{lemma}

\begin{proof}
Solving \eqref{2.45} yields
\begin{equation*}
F^{1/2}(t)\leq F^{1/2}(0)+\frac{1}{2}\int_{0}^{t}\beta(\tau)d\tau.
\end{equation*}
Taking a square of this inequality, one obtains the desired estimate \eqref{Gronwall'}.
\end{proof}

We now derive the $L^2$ energy estimate.

\begin{lemma}\label{L2 estimate}
Let the assumptions of Proposition \ref{A priori estimates} hold. There exist two  constants $C>0$ and $\chi_0>0$ such that if $E(T)\leq\chi_0$, then
\begin{equation}\label{L2}
\begin{aligned}
&\left\|\phi(\cdot,t)
\right\|^2
+\left\|\psi(\cdot,t)
\right\|^2+\int_{0}^{t}
\Big(\||Q_{kx}|^{1/2}\phi(\cdot,\tau) \|^{2}
+\||N_{kt}|^{1/2}\psi(\cdot,\tau) \|^{2}+\left\|\phi_x(\cdot,\tau) \right\|^{2}\Big)\\
&\leq C(\|\phi_0\|^2
+\|\psi_0\|^2
+\varepsilon) \ \text{ for any } t\in[0,T].
\end{aligned}
\end{equation}
\end{lemma}

\begin{proof}
Multiplying the first equation of \eqref{rewrite} by $\phi$ and the
second one by $\psi N_k$, summing them up and integrating the result over $\mathbb{R}\times[0,t]$, we have by integration by parts,
\begin{equation}\label{phi2+psi2}
\begin{aligned}
&\left.\int_{\mathbb{R}}\left(\frac{\phi^2}{2}
+\frac{\psi^2}{2} N_k \right)dx\right|_0^t
-\frac{1}{2}\int_0^t\int_{\mathbb{R}}Q_{kx}\phi^2dxd\tau
-\frac{1}{2}\int_0^t\int_{\mathbb{R}}N_{kt}\psi^2dxd\tau
+\int_0^t\int_{\mathbb{R}}\phi_x^2dxd\tau\\
&=\int_0^t\int_{\mathbb{R}}\phi N_{kxx}dxd\tau
+\int_0^t\int_{\mathbb{R}}\frac{\phi^2}{2} \psi_xdxd\tau.
\end{aligned}
\end{equation}
To estimate the last term on the right hand side (RHS) of \eqref{phi2+psi2}, we rewrite the first equation of \eqref{rewrite} as
\begin{equation}\label{psix}
\begin{aligned}
\psi_x=(N_k+\phi)^{-1}[\phi_t-
(Q_k+\psi)\phi_{x}-N_{kx}\psi-\phi Q_{kx}
-\phi_{xx}-N_{kxx}],
\end{aligned}
\end{equation}
which implies
\begin{equation}\label{4.5}
\begin{aligned}
\int_0^t\int_{\mathbb{R}}\frac{\phi^2}{2} \psi_x
=\int_0^t\int_{\mathbb{R}}(N_k+\phi)^{-1}
\frac{\phi^2}{2}[\phi_t-
Q_k\phi_{x}-\psi\phi_{x}-(N_{kx}\psi+\phi Q_{kx})
-\phi_{xx}-N_{kxx}].
\end{aligned}
\end{equation}

We next estimate each term on the RHS of \eqref{4.5}. Integrating by parts gives
\begin{equation}\label{3.12}
\begin{aligned}
\int_0^t\int_{\mathbb{R}}
\frac{\phi^2}{2}(N_k+\phi)^{-1}\phi_t
=\int_{\mathbb{R}}
\left.(N_k+\phi)^{-1}\frac{\phi^3}{6}\right|
_{\tau=0}^{\tau=t}dx
-\int_0^t\int_{\mathbb{R}}
\partial_t\left((N_k+\phi)^{-1}\right)
\frac{\phi^3}{6}.
\end{aligned}
\end{equation}
By \eqref{3.6} and the monotonicity of $N_k$ in $x$, we get
\[N_k(x,t)+\phi(x,t)\geq
{\color{red}\frac{\underline n_k}{2} }\ \text{ if } E(t)\ll1,\]
which implies the first term on the RHS of \eqref{4.5} satisfies
\begin{equation}\label{I1.1}
\begin{aligned}
\left|\int_{\mathbb{R}}
\left.(N_k+\phi)^{-1}\frac{\phi^3}{6}\right|
_{\tau=0}^{\tau=t}dx\right|
\leq CE(t)\|\phi(t)\|^2+CE(0)\|\phi_0\|^2.
\end{aligned}
\end{equation}
A direct calculation yields
\begin{equation}\label{I1.2}
\begin{aligned}
\left|\int_0^t\int_{\mathbb{R}}
\partial_t\left((N_k+\phi)^{-1}\right)
\frac{\phi^3}{6}\right|
\leq C
\int_0^t\int_{\mathbb{R}}|N_{kt}||\phi|^3
+\left|\int_0^t\int_{\mathbb{R}}
(N_k+\phi)^{-2}\phi_t\frac{\phi^3}{6}\right|.
\end{aligned}
\end{equation}
By Lemma \ref{properties of N,Q}-(4) and \eqref{3.6}, we have
\begin{equation}\label{I1.2.1}
\begin{aligned}
\int_0^t\int_{\mathbb{R}}|N_{kt}||\phi|^3dxd\tau
\leq CE(t)\int_0^t\int_{\mathbb{R}}
|Q_{kx}|\phi^2dxd\tau.
\end{aligned}
\end{equation}
Owing to the first equation of \eqref{rewrite}, we get
\begin{equation}\label{I1.2.2}
\begin{aligned}
\left|\int_0^t\int_{\mathbb{R}}
(N_k+\phi)^{-2}\phi_t\frac{\phi^3}{6}\right|
=&\left|\int_0^t\int_{\mathbb{R}}
(N_k+\phi)^{-2}N_{kxx}
\frac{\phi^3}{6}\right|\\
&+\left|\int_0^t\int_{\mathbb{R}}
(N_k+\phi)^{-2}(\phi\psi+\phi Q_k+N_k\psi+\phi_x)_{x}
\frac{\phi^3}{6}\right|.
\end{aligned}
\end{equation}
It follows from \eqref{3.6} that
\begin{equation}\label{I1.2.2.3}
\begin{aligned}
\left|\int_0^t\int_{\mathbb{R}}
(N_k+\phi)^{-2}N_{kxx}
\frac{\phi^3}{6}\right|
\leq CE(t)\int_0^t\int_{\mathbb{R}}|\phi N_{kxx}|dxd\tau.
\end{aligned}
\end{equation}
Integrating by parts gives rise to
\begin{align}\label{I1.2.2.1'}
&\left|\int_0^t\int_{\mathbb{R}}
(N_k+\phi)^{-2}(\phi\psi+\phi Q_k+N_k\psi+\phi_x)_{x}
\frac{\phi^3}{6}\right| \nonumber\\
&=\left|\int_0^t\int_{\mathbb{R}}
(N_k+\phi)^{-3}(N_{kx}+\phi_x)
(\phi\psi+\phi Q_k+N_k\psi+\phi_x)
\frac{\phi^3}{3}\right.\nonumber\\
&\quad \left.+\int_0^t\int_{\mathbb{R}}
(N_k+\phi)^{-2}(\phi\psi+\phi Q_k+N_k\psi+\phi_x)
\frac{\phi^2}{2}\phi_x\right|\\
&\leq C\int_0^t\int_{\mathbb{R}}
\left|N_{kx}(\phi\psi+\phi Q_k+N_k\psi+\phi_x)
\phi^3\right|
+C\int_0^t\int_{\mathbb{R}}
\left|\phi_x\phi^4\psi+\phi_x\phi^4Q_k+
\phi_xN_k\psi\phi^3
{\color{red}+\phi_x^2\phi^3}\right|\nonumber\\
&\quad+C\int_0^t\int_{\mathbb{R}}
\left|\phi_x\phi^3\psi+\phi_x\phi^3Q_k+
\phi_xN_k\psi\phi^2+\phi_x^2\phi^2
\right|.\nonumber
\end{align}
By Lemma \ref{properties of N,Q}-(4) and \eqref{3.6} again, we get
\begin{equation}\label{I1.2.2.1.1}
\begin{aligned}
\int_0^t\int_{\mathbb{R}}
\left|N_{kx}(\phi\psi+\phi Q_k+N_k\psi+\phi_x)
\phi^3\right|
\leq CE(t)
\int_0^t\int_{\mathbb{R}}|Q_{kx}|\phi^2
+CE(t)\int_0^t\|\phi_x(\tau)\|^2d\tau.
\end{aligned}
\end{equation}
By H\"{o}lder's inequality, Sobolev inequality $\|\phi\|_{L^\infty}\leq \|\phi\|^{1/2}\|\phi_x\|^{1/2}$ and \eqref{3.6}, the second term on the RHS of \eqref{I1.2.2.1'} can be estimated as
\begin{equation}\label{I1.2.2.1.2}
\begin{aligned}
&\int_0^t\int_{\mathbb{R}}
\left|\phi_x\phi^4\psi+\phi_x\phi^4Q_k+
\phi_xN_k\psi\phi^3
{\color{red}+\phi_x^2\phi^3}\right|\\
&\leq CE(t)\int_0^t\int_{\mathbb{R}}
\left|\phi_x\phi^3\right|
{\color{red}+CE(t)
\int_0^t\|\phi_x(\tau)\|^2d\tau}\\
&\leq CE(t)\int_0^t
\|\phi\|^2_{L^\infty}\|\phi\|\|\phi_x\|
{\color{red}+CE(t)
\int_0^t\|\phi_x(\tau)\|^2d\tau}\\
&\leq CE(t)\int_0^t
\|\phi\|^2\|\phi_x\|^2
{\color{red}+CE(t)
\int_0^t\|\phi_x(\tau)\|^2d\tau}\\
&\leq CE(t)
\int_0^t\|\phi_x(\tau)\|^2d\tau,
\end{aligned}
\end{equation}
Similarly, it holds that
\begin{align}\label{I1.2.2.1.3}
&\int_0^t\int_{\mathbb{R}}
\left|\phi_x\phi^3\psi+\phi_x\phi^3Q_k+
\phi_xN_k\psi\phi^2+\phi_x^2\phi^2
\right| \nonumber\\
&\leq CE(t)\int_0^t\int_{\mathbb{R}}
\left|\phi_x\phi^3\right|+
C\int_0^t\int_{\mathbb{R}}
\left|\phi_x\phi^3\right|+
C\int_0^t\int_{\mathbb{R}}
\left|\phi_x\psi\phi^2\right|
+CE(t)
\int_0^t\|\phi_x(\tau)\|^2d\tau\nonumber\\
&\leq CE(t)\int_0^t
\|\phi\|^2_{L^\infty}\|\phi\|\|\phi_x\|
+C\int_0^t
\|\phi\|^2_{L^\infty}\|\psi\|\|\phi_x\|
+CE(t)
\int_0^t\|\phi_x(\tau)\|^2d\tau\\
&\leq CE(t)\int_0^t
\|\phi\|^2\|\phi_x\|^2
+C\int_0^t
\|\phi\|\|\psi\|\|\phi_x\|^2
+CE(t)
\int_0^t\|\phi_x(\tau)\|^2d\tau \nonumber\\
&\leq CE(t)
\int_0^t\|\phi_x(\tau)\|^2d\tau. \nonumber
\end{align}
Substituting \eqref{I1.2.2.1.1}-\eqref{I1.2.2.1.3} into \eqref{I1.2.2.1'}, and combing the result with \eqref{I1.2.2.3}, we have
\begin{equation}\label{I1.2.2.1}
\begin{aligned}
\left|\int_0^t\int_{\mathbb{R}}
(N_k+\phi)^{-2}\phi_t\frac{\phi^3}{6}\right|\leq CE(t)
\int_0^t\int_{\mathbb{R}}|Q_{kx}|\phi^2
+CE(t)\int_0^t\|\phi_x(\tau)\|^2
{\color{red}+CE(t)\int_0^t\int_{\mathbb{R}}|\phi N_{kxx}|}.
\end{aligned}
\end{equation}
Then substituting \eqref{I1.2.1} and \eqref{I1.2.2.1} into \eqref{I1.2}, we arrive at
\begin{equation*}
\begin{aligned}
\left|\int_0^t\int_{\mathbb{R}}
\partial_t\left((N_k+\phi)^{-1}\right)
\frac{\phi^3}{6}\right|
\leq CE(t)
\int_0^t\int_{\mathbb{R}}(|Q_{kx}|\phi^2
+\phi_x^2+|\phi N_{kxx}|),
\end{aligned}
\end{equation*}
which along with \eqref{I1.1} and \eqref{3.12}  gives
\begin{equation}\label{I1}
\begin{aligned}
\int_0^t\int_{\mathbb{R}}
\frac{\phi^2}{2}(N_k+\phi)^{-1}\phi_t
\leq &CE(t)\|\phi(t)\|^2+CE(0)\|\phi_0\|^2
+CE(t)
\int_0^t\int_{\mathbb{R}}|Q_{kx}|\phi^2dxd\tau\\
&+CE(t)\int_0^t\|\phi_x(\tau)\|^2
+CE(t)\int_0^t\int_{\mathbb{R}}|\phi N_{kxx}|dxd\tau.
\end{aligned}
\end{equation}
After integration by parts, by Lemma \ref{properties of N,Q}-(4), one finds that
\begin{equation}\label{I2}
\begin{aligned}
-\int_0^t\int_{\mathbb{R}}(N_k+\phi)^{-1}
\frac{\phi^2}{2}
Q_k\phi_{x}dxd\tau
&=\int_0^t\int_{\mathbb{R}}
\partial_x
\left((N_k+\phi)^{-1}Q_k\right)
\frac{\phi^3}{6}dxd\tau\\
&\leq C
\int_0^t\int_{\mathbb{R}}|Q_{kx}|
|\phi|^3dxd\tau
+C\int_0^t\int_{\mathbb{R}}
|\phi_x||\phi|^3dxd\tau\\
&\leq C
\int_0^t\int_{\mathbb{R}}|Q_{kx}|
|\phi|^3dxd\tau
+C\int_0^t
\|\phi\|^2_{L^\infty}\|\phi\|\|\phi_x\|d\tau\\
&\leq C
\int_0^t\int_{\mathbb{R}}|Q_{kx}|
|\phi|^3dxd\tau
+C\int_0^t
\|\phi\|^2\|\phi_x\|^2d\tau\\
&\leq CE(t)
\int_0^t\int_{\mathbb{R}}|Q_{kx}|\phi^2dxd\tau
+CE(t)\int_0^t\|\phi_x(\tau)\|^2
d\tau,
\end{aligned}
\end{equation}
where we have used Sobolev inequality in the third inequality and \eqref{3.6} in the last inequality.
Noting $N_k(x,t)+\phi(x,t)>
{\color{red}\frac{\underline n_k}{2}}$ when $E(t)\ll1$, we have
\begin{equation}\label{I3}
\begin{aligned}
-\int_0^t\int_{\mathbb{R}}(N_k+\phi)^{-1}
\frac{\phi^2}{2}
\psi\phi_{x}dxd\tau
&\leq C
\int_0^t\int_{\mathbb{R}}
\left|\phi^2\psi\phi_x\right|dxd\tau\\
&\leq C\int_0^t
\|\phi\|^2_{L^\infty}\|\psi\|\|\phi_x\|\\
&\leq C\int_0^t
\|\phi\|\|\psi\|\|\phi_x\|^2\\
&\leq CE(t)
\int_0^t\|\phi_x(\tau)\|^2d\tau,
\end{aligned}
\end{equation}
where we have used \eqref{3.6} in the last inequality.
By Lemma \ref{properties of N,Q}-(4) again, one has
\begin{equation}\label{I4}
\begin{aligned}
-\int_0^t\int_{\mathbb{R}}(N_k+\phi)^{-1}
\frac{\phi^2}{2}
(N_{kx}\psi+\phi Q_{kx})dxd\tau&\leq C\int_0^t\int_{\mathbb{R}}
|Q_{kx}|(|\phi|+|\psi|)
\frac{|\phi|^2}{2}dxd\tau\\&\leq CE(t)
\int_0^t\int_{\mathbb{R}}|Q_{kx}|\phi^2dxd\tau.
\end{aligned}
\end{equation}
A straightforward calculation by integration by parts yields
\begin{equation}\label{I5}
\begin{aligned}
-\int_0^t\int_{\mathbb{R}}(N_k+\phi)^{-1}
\frac{\phi^2}{2}\phi_{xx}dxd\tau
&=\int_0^t\int_{\mathbb{R}}
\partial_x
\left((N_k+\phi)^{-1}\frac{\phi^2}{2}\right)
\phi_xdxd\tau\\
&\leq C\int_0^t\int_{\mathbb{R}}
\left(|\phi|^2|\phi_x||N_{kx}|
+
\phi^2|\phi_x|^2
+
|\phi||\phi_x|^2\right)dxd\tau\\
&\leq CE(t)
\int_0^t\int_{\mathbb{R}}|Q_{kx}|\phi^2dxd\tau
+CE(t)
\int_0^t\|\phi_x(\tau)\|^2d\tau,
\end{aligned}
\end{equation}
where we have used $|N_{kt}|\leq C|Q_{kx}|$ from Lemma \ref{properties of N,Q}-(4). By the fact that $N_k(x,t)+\phi(x,t)>\frac{\underline n_k}{2}$ and \eqref{3.6}, the following estimate holds
\begin{equation}\label{I6}
\begin{aligned}
-\int_0^t\int_{\mathbb{R}}(N_k+\phi)^{-1}
\frac{\phi^2}{2}N_{kxx}\leq C\int_0^t\int_{\mathbb{R}}
|\phi|^2|N_{kxx}| \leq CE(t)\int_0^t\int_{\mathbb{R}}|\phi N_{kxx}|.
\end{aligned}
\end{equation}
Substituting \eqref{I1}-\eqref{I6} into \eqref{psix} leads to
\begin{equation}\label{phi2 psix}
\begin{aligned}
\left|\int_0^t\int_{\mathbb{R}}\frac{\phi^2}{2} \psi_xdxd\tau\right|
\leq &CE(t)\|\phi(t)\|^2+CE(0)\|\phi_0\|^2
+CE(t)\int_0^t\int_{\mathbb{R}}|\phi N_{kxx}|dxd\tau\\
&+CE(t)\int_0^t\int_{\mathbb{R}}|Q_{kx}|\phi^2dxd\tau
+CE(t)
\int_0^t\|\phi_x(\tau)\|^2d\tau.
\end{aligned}
\end{equation}
Combining \eqref{phi2+psi2} and \eqref{phi2 psix}, we have
\begin{equation}\label{4.1}
\begin{aligned}
&\|\phi(\cdot,t)\|^2
+\|\psi(\cdot,t)\|^2
+\int_0^t\int_{\mathbb{R}}|Q_{kx}|\phi^2
+\int_0^t\int_{\mathbb{R}}|N_{kt}|\psi^2
+\int_0^t\|\phi_x(\tau)\|^2
\\ &\leq \fix{C(\|\phi_0\|^2
+\|\psi_0\|^2)}
+C\int_0^t\int_{\mathbb{R}}|\phi N_{kxx}|dxd\tau\\
&\leq \fix{C(\|\phi_0\|^2
+\|\psi_0\|^2)}
+C\int_0^t
\|\phi(\cdot,\tau)\|\|N_{kxx}(\cdot,\tau)\|d\tau,
\end{aligned}
\end{equation}
provided that $E(t)$ is suitably small. Now set
\begin{equation}\label{3.35}
\begin{aligned}
F(t)\triangleq \fix{C(\|\phi_0\|^2
+\|\psi_0\|^2)}+C\int_0^t
\|\phi(\cdot,\tau)\|\|N_{kxx}(\cdot,\tau)\|d\tau, \ \beta(t)\triangleq C\|N_{kxx}(\cdot,t)\|.
\end{aligned}
\end{equation}
It follows from \eqref{4.1} that
\[(F'(t))^2\leq\beta^2(t)F(t), \ \forall t\in[0,T].\]
Therefore, by \eqref{Nxx} and Lemma \ref{Gronwall},
we obtain the
estimate \eqref{L2} and complete the proof.
\end{proof}

We next derive the estimate for $\|\psi_x\|$ .

\begin{lemma}\label{H1psix}
Let the assumptions of Proposition \ref{A priori estimates} hold. There exist two  constants $C>0$ and $\chi_0>0$ such that if $E(T)\leq\chi_0$, then
\begin{equation}\label{H1.1}
\begin{aligned}
\|\psi_x(\cdot,t)\|^2
+\int_0^t\|
\psi_x(\cdot,\tau)\|^2
\leq &C(\|\phi_0\|^2
+\|\psi_0\|_1^2
+\varepsilon) \ \text{ for any } t\in[0,T].
\end{aligned}
\end{equation}

\end{lemma}

\begin{proof}
Substituting $\psi_t=\phi_x$ into the first equation of \eqref{rewrite} yields
\begin{equation}\label{psitx}
\begin{aligned}
\psi_{xt}=\phi_t -(\phi\psi+\phi Q_k+N_k\psi)_x-N_{kxx}.
\end{aligned}
\end{equation}
Multiplying \eqref{psitx} by $\psi_x$ and integrating the result over $\mathbb{R}\times[0,t]$, noting \begin{equation*}
\begin{aligned}
\phi_t\psi_x&=(\phi\psi_x)_t-\phi\psi_{xt}
=(\phi\psi_x)_t-\phi\phi_{xx}
=(\phi\psi_x)_t-(\phi\phi_x)_x+\phi_x^2,
\end{aligned}
\end{equation*}
we have
\begin{equation}\label{psix2}
\begin{aligned}
&\left.\left(\int_{\mathbb{R}}\frac{\psi_x^2}{2}
- \phi\psi_x dx\right)\right|_0^t
+\int_0^t\int_{\mathbb{R}}        N_k\psi_x^2dxd\tau\\
=&\int_0^t\int_{\mathbb{R}}\phi_x^2dxd\tau
-\int_0^t\int_{\mathbb{R}}\psi_{x}(\phi Q_k)_{x} dxd\tau
-\int_0^t\int_{\mathbb{R}}\psi_{x}(\phi\psi)_{x} dxd\tau\\
&+\int_0^t\int_{\mathbb{R}}\frac{\psi^2}{2}N_{kxx} dxd\tau
-\int_0^t\int_{\mathbb{R}}\psi_{x}N_{kxx} dxd\tau.
\end{aligned}
\end{equation}
We next estimate each term on the RHS of \eqref{psix2}. {\color{red}A direct calculation yields
\begin{equation}\label{'J2}
-\int_0^t\int_{\mathbb{R}}\psi_x(\phi Q_k)_x dxd\tau
=-\int_0^t\int_{\mathbb{R}}\psi_x\phi_xQ_kdxd\tau
-\int_0^t\int_{\mathbb{R}}\psi_x\phi Q_{kx}dxd\tau.
\end{equation}
For the first term on the RHS of \eqref{J2}, by the Peter-Paul inequality and \eqref{Nk}, we obtain
\begin{equation}\label{J21}
\begin{aligned}
\left|
\int_0^t\int_{\mathbb R}\psi_x\phi_xQ_k dx d\tau
\right|
&\leq
\frac{\gamma_k}{2}
\int_0^t\int_{\mathbb R}
|Q_k||\psi_x|^2\,dx d\tau
+\frac{1}{2\gamma_k}
\int_0^t\int_{\mathbb R}
|Q_k||\phi_x|^2dx d\tau
\\
&\leq
\frac12
\int_0^t\int_{\mathbb R}
N_k|\psi_x|^2dx d\tau
+C\int_0^t\|\phi_x(\tau)\|^2d\tau,
\end{aligned}
\end{equation}
with $\gamma_k:=\frac{\inf_{x,t}N_k(x,t)}
{\|Q_k\|_{L^\infty}}$.
It follows from Young's inequality and Lemma \ref{properties of N,Q}-(2): $\|Q_{kx}\|_{L^\infty}\leq C\varepsilon$ that
\begin{equation}\label{J22}
\begin{aligned}
\left|
\int_0^t\int_{\mathbb R}
\psi_x\phi Q_{kx}dx d\tau
\right|
&\leq
\frac14
\int_0^t\int_{\mathbb R}
|Q_{kx}||\phi|^2dx d\tau
+C\varepsilon
\int_0^t\|\psi_x(\tau)\|^2d\tau.
\end{aligned}
\end{equation}
Substituting \eqref{J21} and \eqref{J22} into \eqref{'J2} leads to
\begin{equation}\label{J2}
\begin{aligned}
-\int_0^t\int_{\mathbb R}
\psi_x(\phi Q_k)_x dx d\tau
\leq&
\frac12
\int_0^t\int_{\mathbb R}
N_k|\psi_x|^2 dx d\tau
+C\int_0^t\|\phi_x(\tau)\|^2d\tau
\\
&+
\frac14
\int_0^t\int_{\mathbb R}
|Q_{kx}||\phi|^2dx d\tau
+C\varepsilon
\int_0^t\|\psi_x(\tau)\|^2d\tau.
\end{aligned}
\end{equation}
}
Using Young's inequality
and  \eqref{3.6}, we have
\begin{equation}\label{J3'}
\begin{aligned}
-\int_0^t\int_{\mathbb{R}}\psi_{x}(\phi \psi)_{x} dxd\tau
&=-\int_0^t\int_{\mathbb{R}}\psi_{x}\phi_{x}\psi dxd\tau-\int_0^t\int_{\mathbb{R}}\phi\psi_{x}^2
dxd\tau\\&\leq CE(t) \int_0^t
\left(\|\psi_x\|^{2}
+\|\phi_x\|^2\right)d\tau+CE(t) \int_0^t\|\psi_x\|^2d\tau.
\end{aligned}
\end{equation}
Thanks to the Sobolev inequality and \eqref{3.6}, we get
\begin{equation}\label{J4'}
\begin{aligned}
\int_0^t\int_{\mathbb{R}}\frac{\psi^2}{2}N_{kxx} dxd\tau
&\leq C\int_0^t
\|\psi\|\|\psi_x\|\|N_{kxx}\|_{L_1} d\tau\\
&\leq CE\int_0^t\|\psi_x(\tau)\|^2d\tau
+CE\int_0^t\|N_{kxx}\|_{L_1}^2d\tau.
\end{aligned}
\end{equation}
{\color{red}
The pointwise estimate \eqref{Ul} with $p=1$ gives $\|N_{kxx}\|_{L_1}\leq C \min\{\varepsilon,\ t^{-1}\}$, which implies
\begin{equation}\label{NkxxL1}
\begin{aligned}
\int_0^t\|N_{kxx}\|_{L_1}^2d\tau
\leq C\int_0^\infty \min\{\varepsilon^2,\ t^{-2}\}d\tau
\leq C\varepsilon.
\end{aligned}
\end{equation}
}
Substituting \eqref{NkxxL1} into the RHS  of \eqref{J4'} leads to
\begin{equation}\label{J4}
\begin{aligned}
\int_0^t\int_{\mathbb{R}}\frac{\psi^2}{2}N_{kxx} dxd\tau
&\leq CE(t) \int_0^t\|\psi_x(\tau)\|^2d\tau
+C E(t)\varepsilon.
\end{aligned}
\end{equation}

Therefore, combining the estimates \eqref{J2}, \eqref{J3'}, \eqref{J4} and Lemma \ref{L2 estimate}, we arrive at
\begin{equation}\label{psix2'}
\begin{aligned}
\|\psi_x(\cdot,t)\|^2
+\int_0^t\|
\psi_x(\cdot,\tau)\|^2
\leq &C(\|\phi_0\|^2
+\|\psi_0\|^2+\|\psi_{0x}\|^2)
+C E\varepsilon
+\int_0^t\int_{\mathbb{R}}
|\psi_{x}N_{kxx}|,
\end{aligned}
\end{equation}
provided that $E(t)$ is suitably small. As in \eqref{3.35}, by Lemma \ref{Gronwall} and \eqref{Nxx}, we obtain the estimate \eqref{H1.1}.
\end{proof}

We are now ready to establish the estimate for $\|\phi_x\|$.

\begin{lemma}\label{H1phix}
Let the assumptions of Proposition \ref{A priori estimates} hold. There exist two  constants $C>0$ and $\chi_0>0$ such that if $E(T)\leq\chi_0$, then for any $t\in[0,T]$ it holds that
\begin{equation}\label{H1.2}
\begin{aligned}
\|\phi_x(\cdot,t)\|^2
+\int_0^t\|\phi_{xx}(\cdot,\tau)\|^2
\leq &C(\|\phi_0\|_1^2
+\|\psi_0\|_1^2
+\varepsilon).
\end{aligned}
\end{equation}

\end{lemma}

\begin{proof}
Multiplying the first equation of \eqref{rewrite} by $-\phi_{xx}$ and integrating the result
over $\mathbb{R}\times[0,t] $, we have
\begin{equation}\label{phixx}
\begin{aligned}
\left.\int_{\mathbb{R}}\frac{\phi_x^2}{2}
dx\right|_0^t
+\int_0^t\int_{\mathbb{R}} \phi_{xx}^2
=&-\int_0^t\int_{\mathbb{R}}
\phi_{xx}\left[(\phi\psi+\phi Q_k+N_k\psi)_{x}
+N_{kxx}\right]\\
\leq &\frac{1}{2}\int_0^t\int_{\mathbb{R}}
\phi_{xx}^2
+C\int_0^t\int_{\mathbb{R}}
(|Q_k||\phi_x|^2+|Q_{kx}|^2|\phi|^2+|N_{kx}|^2|\psi|^2
\\&\quad+|N_k||\psi_x|^2
+|\phi_x|^2|\psi|^2+|\phi|^2|\psi_x|^2
)
+\frac{1}{2}\int_0^t\int_{\mathbb{R}}
|N_{kxx}|^2.
\end{aligned}
\end{equation}
By the Sobolev inequality, Young's inequality and \eqref{3.6}, we get
\begin{equation}\label{K3}
\begin{aligned}
\int_0^t\int_{\mathbb{R}}
|N_{kx}|^2|\psi|^2dxd\tau
\leq &C\int_0^t
\|\psi\|\|\psi_x\|
\|N_{kx}\|^2d\tau\\
\leq &C\int_0^t\|\phi\|
\left(\|\psi_x\|^2+\|N_{kx}\|^4\right)d\tau\\
\leq &CE(t)\int_0^t\|\psi_x\|^2d\tau
+CE(t)\int_0^t\min(\varepsilon^2, \tau^{-2})d\tau\\
\leq &CE(t) \int_0^t
\|\psi_x\|^2d\tau+CE(t)\varepsilon,
\end{aligned}
\end{equation}
where we have used \eqref{Ux} of Lemma \ref{properties of N,Q} in the third inequality.
By H\"{o}lder's inequality, we further have
\begin{equation}\label{K5}
\begin{aligned}
\int_0^t\int_{\mathbb{R}}
|\phi_x|^2|\psi|^2dxd\tau
\leq C \|\psi\|_{L^\infty}^2
\int_0^t\int_{\mathbb{R}}
|\phi_x|^2d\tau
\leq CE^2\int_0^t
\|\phi_x\|^2d\tau,
\end{aligned}
\end{equation}
and
\begin{equation}\label{K6}
\begin{aligned}
\int_0^t\int_{\mathbb{R}}
|\phi|^2|\psi_x|^2dxd\tau
\leq C \|\phi\|_{L^\infty}^2
\int_0^t\int_{\mathbb{R}}
|\psi_x|^2d\tau
\leq CE^2\int_0^t
\|\psi_x\|^2d\tau.
\end{aligned}
\end{equation}
Owing to \eqref{Ul} {\color{red} with $p=2$}, it holds
\begin{equation}\label{K7}
\begin{aligned}
\frac{1}{2}\int_0^t
\|N_{kxx}\|^2d\tau
\leq C\int_0^\infty
\min\{\varepsilon^3,{\color{red}
\varepsilon^{(1-\frac1{2\theta})(1-\frac1p)}
\tau^{-2-\frac{1}{2\theta}}}\}d\tau
\leq C\varepsilon^2.
\end{aligned}
\end{equation}
Now substituting \eqref{K3}-\eqref{K7} into \eqref{phixx}, by
Lemmas \ref{L2 estimate} and \ref{H1psix}, we obtain the desired estimate \eqref{H1.2}.
\end{proof}

\begin{proof}[Proof of Proposition \ref{A priori estimates}]
It is a consequence of Lemmas \ref{L2 estimate}, \ref{H1psix} and \ref{H1phix}.
\end{proof}

We are now in a position to prove Proposition \ref{phi stability}.
\begin{proof}[Proof of Proposition \ref{phi stability}]
We define 
\begin{equation}\label{deta0}
\begin{aligned}
\Xi_0:=\frac{\chi_0}{4},
\ \delta_1:=
\frac{1}{4}\sqrt{\frac{\chi_0^2}{C_0}-16\varepsilon},
\end{aligned}
\end{equation}
where $\chi_0\leq1$ and $C_0>1$ are constants given in Proposition \ref{A priori estimates}, and  $\varepsilon$ is suitably small.

\emph{Step 1.} Since
$\|\phi_0,\psi_0\|_1 \leq \delta_1\leq\frac{\chi_0}{4}$, by the local existence result established in Proposition \ref{local existence}, there is a positive constant $T_0=T_0(\chi_0)$ such that the system \eqref{rewrite}-\eqref{phi0,psi0} has a unique solution on $[0,T_0]$, which satisfies
\begin{equation}\label{3.52}
\|(\phi,\psi)\|_1\leq 2\|\phi_0,\psi_0\|_1
\leq \frac{\chi_0}{2}<\chi_0 \text{ for } t \in [0,T_0].\end{equation}
Applying the \emph{a priori} estimate established in Proposition \ref{A priori estimates} with $T=T_0$, we get from \eqref{deta0} that
\begin{equation*}
\begin{aligned}
\|(\phi,\psi)(T_0)\|_1\leq
\sqrt{C_0}
\sqrt{\|\phi_0,\psi_0\|_1^2+\varepsilon}
\leq \sqrt{C_0}\sqrt{\delta_1^2+\varepsilon}
\leq \frac{\chi_0}{4}=\Xi_0.
\end{aligned}
\end{equation*}

\emph{Step 2.}
Taking $t=T_0$ as the new initial time and applying Proposition \ref{local existence}, one can see that the system \eqref{rewrite}-\eqref{phi0,psi0} has a unique solution on $[T_0,2T_0]$, which satisfies
\begin{equation*}
\begin{aligned}
\|(\phi,\psi)\|_1\leq 2\|(\phi,\psi)(T_0)\|_1
\leq \frac{\chi_0}{2}
< \chi_0
\ \text {for} \ t \in [T_0,2T_0].
\end{aligned}
\end{equation*}
This along with \eqref{3.52} gives rise to $\|(\phi,\psi)\|_1
\leq \chi_0$ for $t \in [0,2T_0]$. Hence, applying Proposition \ref{A priori estimates} with $T=2T_0$, we get
\begin{equation*}
\begin{aligned}
\|(\phi,\psi)\|_1\leq
\sqrt{C_0}
\sqrt{\|\phi_0,\psi_0\|_1^2+\varepsilon} \text{ for }t \in [0, 2T_0].
\end{aligned}
\end{equation*}
Now it follows from \eqref{deta0} that
\begin{equation*}
\begin{aligned}
\|(\phi,\psi)(2T_0)\|_1\leq
\sqrt{C_0}
\sqrt{\|\phi_0,\psi_0\|_1^2+\varepsilon}
\leq \sqrt{C_0}\sqrt{\delta_0^2+\varepsilon}
\leq \frac{\chi_0}{4}=\Xi_0.
\end{aligned}
\end{equation*}
When $\chi_0\ll 1$, $N_k+\phi$ has a lower bound
$$N_k+\phi>
{\color{red}\frac{\underline n_k}{2}} \ \text{ for } t\in[0,2T_0],$$
which ensures the feasibility of successive extensions. Then by repeating this continuation process, we can extend the solution to the whole time
interval $[0, \infty)$ successively. Moreover, the solution satisfies
\begin{equation}\label{global}
\begin{aligned}
&\sup_{0\leq \tau< \infty}
\left\|(\phi,\psi)(\tau,\cdot)
\right\|_1^2
+\int_{0}^{\infty}
\left(\left\||Q_{kx}|^{1/2}\phi(\tau) \right\|^{2}
+\left\||N_{k\tau}|^{1/2}\psi(\tau) \right\|^{2}+\left\|\phi_x(\tau) \right\|_1^{2}+\left\|\psi_x(\tau) \right\|^{2}\right)d\tau\\
&\leq C (\|(\phi_0,\psi_0)\|_1^2+\varepsilon).
\end{aligned}
\end{equation}

\emph{Step 3.}
It remains to show the $L^\infty$ convergence \eqref{asymptotic}.  We first deduce from \eqref{global} that
\begin{equation}\label{phix2+psix2}
\begin{aligned}
\int_{0}^{+\infty}\int_{\mathbb{R}}
\left(\phi_x^2+\psi_x^2\right)dx
\leq C (\|(\phi_0,\psi_0)\|_1^2+\varepsilon)
<+\infty.
\end{aligned}
\end{equation}
We next show $\int_{0}^{+\infty}
\left|\frac{d}{dt}\int_{\mathbb{R}}
\left(\phi_x^2+\psi_x^2\right)dx\right|d\tau
<+\infty.$ A straightforward calculation yields
\begin{equation}\label{phix2t+psix2t}
\begin{aligned}
\int_{0}^{+\infty}
\left|\frac{d}{dt}\int_{\mathbb{R}}
\left(\phi_x^2+\psi_x^2\right)dx\right|d\tau
&=2\int_{0}^{+\infty}\left|\int_{\mathbb{R}}
\phi_x\phi_{xt}+\psi_x\psi_{xt}dx\right| d\tau\\
&\leq C\int_{0}^{+\infty}\left|\int_{\mathbb{R}}
\phi_x\phi_{xt}dx\right|d\tau
+C\int_{0}^{+\infty}\left|\int_{\mathbb{R}}
\psi_x\psi_{xt}dx\right|d\tau.
\end{aligned}
\end{equation}
Integration by parts and using the first equation of \eqref{rewrite}, we have
\begin{equation*}
\begin{aligned}
\int_{0}^{+\infty}\left|\int_{\mathbb{R}}
\phi_x\phi_{xt}dx\right|d\tau
&=\int_{0}^{+\infty}\left|\int_{\mathbb{R}}
\phi_{xx}\phi_tdx\right|d\tau\\
&=\int_{0}^{+\infty}\left|\int_{\mathbb{R}}
\phi_{xx}(\phi\psi+\phi Q_k+N_k\psi)_{x}
+\phi_{xx}^2+\phi_{xx}N_{kxx}dx\right|d\tau,
\end{aligned}
\end{equation*}
which, in combination with \eqref{3.6}, leads to
\begin{equation}\label{phix2t}
\begin{aligned}
\int_{0}^{+\infty}\left|\int_{\mathbb{R}}
\phi_x\phi_{xt}
\right|d\tau
&\leq CE \int_{0}^{+\infty}\int
\left(|Q_{kx}|\phi^2+|N_{kt}|\psi^{2}+\phi_x^{2}+\psi_x^{2}\right)\\
&\quad+C\int_{0}^{+\infty}
\left\|\phi_{xx}(\tau) \right\|^{2}d\tau
+\frac{1}{2}\int_{0}^{+\infty}
\left\|N_{kxx}\right\|^{2}d\tau\\
&\leq C (\|(\phi_0,\psi_0)\|_1^2+\varepsilon),
\end{aligned}
\end{equation}
where we have used $\int_{0}^{+\infty}
\left\|N_{kxx}\right\|^{2}d\tau\leq C\varepsilon^3$ from \eqref{Ul}.
By the second equation of \eqref{rewrite}, we get
\begin{equation}\label{psix2t}
\begin{aligned}
\int_{0}^{+\infty}\left|\int_{\mathbb{R}}
\psi_x\psi_{xt}dx\right|d\tau
&=\int_{0}^{+\infty}\left|\int_{\mathbb{R}}
\psi_x\phi_{xx}dx\right|d\tau\\
&\leq C\int_{0}^{+\infty}
\left\|\psi_x(\tau)\right\|^{2}
+\left\|\phi_{xx}(\tau) \right\|^{2}
d\tau\\
&\leq C (\|(\phi_0,\psi_0)\|_1^2+\varepsilon).
\end{aligned}
\end{equation}
Substituting \eqref{phix2t} and \eqref{psix2t} into \eqref{phix2t+psix2t} leads to
\begin{equation*}
\begin{aligned}
\int_{0}^{+\infty}
\left|\frac{d}{dt}\int_{\mathbb{R}}
\left(\phi_x^2+\psi_x^2\right)dx\right|d\tau
<\infty.
\end{aligned}
\end{equation*}
This together with \eqref{phix2+psix2} gives rise to
\begin{equation*}
\begin{aligned}
\int_{0}^{+\infty}\left[\int_{\mathbb{R}}
\left(\phi_x^2+\psi_x^2\right)dx+
\left|\frac{d}{dt}\int_{\mathbb{R}}
\left(\phi_x^2+\psi_x^2\right)dx\right|\right]d\tau
<+\infty.
\end{aligned}
\end{equation*}
Hence
\begin{equation*}
\lim\limits_{t\rightarrow \infty}
\left\|(\phi_x,\psi_x)(\cdot,t)\right\|^2=0.
\end{equation*}
It then follows from the Sobolev inequality that
\begin{equation}\label{Sobolev's inequality}
\begin{split}
\sup_{x\in \mathbb{R}}
\left|(\phi,\psi)(x,t)\right|^2
&\leq 2
(\left\|\phi(\cdot,t)\right\| \left\|\phi_x(\cdot,t)\right\|+\left\|\psi(\cdot,t)\right\| \left\|\psi_x(\cdot,t)\right\|)\\& \leq C(\left\|\phi_x(\cdot,t)\right\|+\left\|\psi_x(\cdot,t)\right\|)\rightarrow0 \text{ as } t\rightarrow\infty.
\end{split}\end{equation} This completes the proof of Proposition \ref{phi stability}.
\end{proof}

\begin{proof}[Proof of Theorem \ref{the1.1}]
Recalling $(n-N_k,q-Q_k)=(\phi,\psi)$, by Proposition \ref{phi stability}, one  can see that if the initial value $(n_0,q_0)(x)$ satisfies \eqref{n0-n0r}, then the system \eqref{1.1}-\eqref{i d} has a unique global solution satisfying \eqref{solutio}. Furthermore, by Lemma \ref{lem2.1}-(2) and \eqref{Sobolev's inequality}, $(n,q)$ has the asymptotic behavior \eqref{approximate}. We complete the proof of Theorem \ref{the1.1}.
\end{proof}

\section{Stability of composite  rarefaction waves}\label{Sect.4}
In this section, we study the asymptotic stability of composite rarefaction waves and prove Theorem \ref{the1.2}.
By decomposing $(n,q)=(N+\Phi,Q+\Psi)$, where $(N,Q)$ is the smooth approximation of the composite rarefaction wave constructed in  Section 2, we get the following equation of $(\Phi,\Psi)$
\begin{equation}\label{rewrite'}
\left\{
\begin{aligned}
&\Phi_{t} -(\Phi\Psi+\Phi Q+N\Psi)_{x}=\Phi_{xx}+N_{xx}+g(N,Q)_x,\\
&\Psi_{t}-\Phi_{x}=0,
\end{aligned}
\right.
\end{equation}
with initial value
\begin{equation}\label{Phi0,Psi0}
\begin{aligned}
(\Phi,\Psi)(x,0):=(\Phi_0,\Psi_0)(x)
=(n_0-N(x,0),q_0-Q(x,0))\in H^1.
\end{aligned}
\end{equation}
As in the single wave case, we also search for solutions of \eqref{rewrite'}-\eqref{Phi0,Psi0} in the function space $X(0,+\infty)$.

\begin{proposition}[Global existence]\label{Global existence}
Fix the admissible end states and $\theta>3/2$. There exist $\varepsilon_0,\delta_1>0$, independent of $0<\varepsilon\leq\varepsilon_0$, such that if
$\|(\Phi_0,\Psi_0)\|_1\leq \delta_1$, then the system \eqref{rewrite'}-\eqref{Phi0,Psi0} has a unique global solution $(\Phi,\Psi)\in X(0,+\infty)$ satisfying
\begin{equation}\label{global estimates}
\begin{aligned}
\sup_{0\leq \tau\leq t}
\left\|(\Phi,\Psi)(\cdot,t)
\right\|_1^2
+\int_{0}^{t}
&\Big(\left\||Q_{x}|^{1/2}\Phi(\cdot,\tau) \right\|^{2}
+\left\||N_{t}|^{1/2}\Psi(\cdot,\tau) \right\|^{2}\\&+\left\|\Phi_x(\cdot,\tau) \right\|_1^{2}+\left\|\Psi_x(\cdot,\tau) \right\|^{2}\Big)d\tau\leq C (\|(\Phi_0,\Psi_0)\|_1^2+\varepsilon)
\end{aligned}
\end{equation}
for any $t >0$.
\end{proposition}

As in Proposition \ref{local existence}, the local existence of solutions to the system \eqref{rewrite'}-\eqref{Phi0,Psi0} is standard. To prove Proposition \ref{global estimates} we only need to establish the following \emph{a priori} estimate.

\begin{proposition}[\emph{A priori} estimate]\label{'A priori estimates'}
Fix the admissible end states and $\theta>3/2$.  Suppose that the system \eqref{rewrite'}-\eqref{Phi0,Psi0} has a solution $(\Phi,\Psi)\in X(0,T)$ for some $T>0$.
Then there exists a positive constants $\chi_1\ll 1$ independent of $T$ such that if
\begin{equation}\label{assumption'}
E(t):=\sup\limits_{0\leq t\leq T}\|(\Phi,\Psi)(\cdot,t)\|_1
\leq \chi_1,
\end{equation}
then the estimate \eqref{global estimates} holds for any $t\in[0,T]$.

\end{proposition}

To prove Proposition \ref{'A priori estimates'}, we first derive the $L^2$ estimate.
\begin{lemma}\label{L2 estimate'}
Let the assumptions of Proposition \ref{'A priori estimates'} hold.
If $\chi_1\ll1$, then there exists a constant $C>0$ such that
\begin{equation}\label{L2'}
\begin{aligned}
&\|\Phi(\cdot,t)\|^2
+\|\Psi(\cdot,t)\|^2
+\int_0^t\int_{\mathbb{R}}|Q_{x}|\Phi^2dxd\tau
+\int_0^t\int_{\mathbb{R}}|N_{t}|\Psi^2dxd\tau
+\int_0^t\|\Phi_x(\cdot,\tau)\|^2d\tau\\
&\leq C(\|\Phi_0\|^2
+\|\Psi_0\|^2
+\varepsilon) \ \text{ for }t\in[0,T].
\end{aligned}
\end{equation}
\end{lemma}

\begin{proof}
We multiply the first equation of \eqref{rewrite'} by $\Phi$ and the
second one by $\Psi N$, sum them up and integrate it over $\mathbb{R}\times [0,t]$ to have
\begin{equation}\label{Phi2+Psi2}
\begin{aligned}
&\left.\left(\int_{\mathbb{R}}\frac{\Phi^2}{2}
+\frac{\Psi^2}{2} Ndx\right)\right|_0^t
-\frac{1}{2}\int_0^t\int_{\mathbb{R}}Q_{x}\Phi^2dxd\tau
-\frac{1}{2}\int_0^t\int_{\mathbb{R}}N_{t}\Psi^2dxd\tau
+\int_0^t\int_{\mathbb{R}}\Phi_x^2dxd\tau\\
&=\int_0^t\int_{\mathbb{R}}\Phi N_{xx}dxd\tau
+\int_0^t\int_{\mathbb{R}}\Phi g(N,Q)_xdxd\tau
+\int_0^t\int_{\mathbb{R}}\frac{\Phi^2}{2} \Psi_xdxd\tau.
\end{aligned}
\end{equation}
To estimate the last term of \eqref{Phi2+Psi2}, we rewrite the first equation of \eqref{rewrite'} as
\begin{equation}\label{Psix}
\begin{aligned}
\Psi_x=(N+\Phi)^{-1}[\Phi_t-
(Q+\Psi)\Phi_{x}-N_x\Psi-\Phi Q_x
-\Phi_{xx}-N_{xx}-g(N,Q)_x].
\end{aligned}
\end{equation}
It then follows that
\begin{equation}\label{Psix'}
\begin{aligned}
\int_0^t\int_{\mathbb{R}}\frac{\Phi^2}{2} \Psi_xdxd\tau
&=\int_0^t\int_{\mathbb{R}}(N+\Phi)^{-1}
\frac{\Phi^2}{2}\left[\Phi_t-
Q\Phi_{x}-\Psi\Phi_{x}-(N_x\Psi+\Phi Q_x)
-\Phi_{xx}\right.\\
&~~~~\left.-N_{xx}-g(N,Q)_x\right]dxd\tau.
\end{aligned}
\end{equation}

We next estimate each term on the RHS of \eqref{Psix'}. Integrating by parts gives
\begin{equation}\label{4.9}
\begin{aligned}
\int_0^t\int_{\mathbb{R}}
\frac{\Phi^2}{2}(N+\Phi)^{-1}\Phi_t
=\int_{\mathbb{R}}
\left.(N+\Phi)^{-1}\frac{\Phi^3}{6}\right|
_{\tau=0}^{\tau=t}dx
-\int_0^t\int_{\mathbb{R}}
\partial_t\left((N+\Phi)^{-1}\right)
\frac{\Phi^3}{6}.
\end{aligned}
\end{equation}
As in \eqref{3.6}, by \eqref{assumption'} and the Sobolev embedding theorem, we have
\begin{equation}\label{4.10}
\sup\limits_{0\leq \tau\leq t}(\|\Phi(\cdot,\tau)\|_{L^\infty}+\|\Psi(\cdot,\tau)\|_{L^\infty})\leq CE(t).
\end{equation}
By {\color{red}Lemma \ref{properties of N',Q'}-(1): $N(x,t)\geq\bar n>0$}, we get
\[N(x,t)+\Phi(x,t)>{\color{red}\frac{\bar n}{2}}, \text{ if } E(t)\ll1.\]
It then follows
\begin{equation}\label{J1.1}
\begin{aligned}
\left|\int_{\mathbb{R}}
\left.(N+\Phi)^{-1}\frac{\Phi^3}{6}\right|
_{\tau=0}^{\tau=t}dx\right|
\leq CE(t)\|\Phi(t)\|^2+CE(0)\|\Phi_0\|^2.
\end{aligned}
\end{equation}
A direct calculation yields
\begin{equation}\label{J1.2}
\begin{aligned}
\left|\int_0^t\int_{\mathbb{R}}
\partial_t\left((N+\Phi)^{-1}\right)
\frac{\Phi^3}{6}\right|
\leq C
\int_0^t\int_{\mathbb{R}}|N_t||\Phi|^3
+\left|\int_0^t\int_{\mathbb{R}}
(N+\Phi)^{-2}\Phi_t\frac{\Phi^3}{6}\right|.
\end{aligned}
\end{equation}
Owing to Lemma \ref{properties of N',Q'}-(4), it holds
that
\begin{equation}\label{J1.2.1}
\begin{aligned}
\int_0^t\int_{\mathbb{R}}|N_t||\Phi|^3dxd\tau
\leq CE(t)\int_0^t\int_{\mathbb{R}}
|Q_x|\Phi^2dxd\tau.
\end{aligned}
\end{equation}
Using the first equation of \eqref{rewrite'}, we get
\begin{equation}\label{J1.2.2}
\begin{aligned}
\left|\int_0^t\int_{\mathbb{R}}
(N+\Phi)^{-2}\Phi_t\frac{\Phi^3}{6}\right|
\leq&\left|\int_0^t\int_{\mathbb{R}}
(N+\Phi)^{-2}(\Phi\Psi+\Phi Q+N\Psi+\Phi_x)_{x}
\frac{\Phi^3}{6}\right|\\
&+\left|\int_0^t\int_{\mathbb{R}}
(N+\Phi)^{-2}(N_{xx}+g(N,Q)_x)
\frac{\Phi^3}{6}\right|.
\end{aligned}
\end{equation}
By \eqref{3.6}, we derive
\begin{equation}\label{J1.2.2.3}
\begin{aligned}
\left|\int_0^t\int_{\mathbb{R}}
(N+\Phi)^{-2}(N_{xx}+g(N,Q)_x)
\frac{\Phi^3}{6}\right|
\leq CE(t)\int_0^t\int_{\mathbb{R}}
|\Phi (N_{xx}+g(N,Q)_x)|dxd\tau.
\end{aligned}
\end{equation}
Integrating by parts gives rise to
\begin{equation}\label{J1.2.2.1'}
\begin{aligned}
&\left|\int_0^t\int_{\mathbb{R}}
(N+\Phi)^{-2}(\Phi\Psi+\Phi Q+N\Psi+\Phi_x)_{x}
\frac{\Phi^3}{6}\right|\\
&=\left|\int_0^t\int_{\mathbb{R}}
(N+\Phi)^{-3}(N_x+\Phi_x)
(\Phi\Psi+\Phi Q+N\Psi+\Phi_x)
\frac{\Phi^3}{3}\right.\\
&~~~\left.+\int_0^t\int_{\mathbb{R}}
(N+\Phi)^{-2}(\Phi\Psi+\Phi Q+N\Psi+\Phi_x)
\frac{\Phi^2}{2}\Phi_x\right|\\
&\leq C\int_0^t\int_{\mathbb{R}}
\left|N_x(\Phi\Psi+\Phi Q+N\Psi+\Phi)
\Phi^3\right|
+C\int_0^t\int_{\mathbb{R}}
\left|\Phi_x\Phi^4\Psi+\Phi_x\Phi^4Q+
\Phi_xN\Psi\Phi^3\right|\\
&~~~+C\int_0^t\int_{\mathbb{R}}
\left|\Phi_x\Phi^3\Psi+\Phi_x\Phi^3Q+
\Phi_xN\Psi\Phi^2+\Phi_x^2\Phi^2+\Phi_x^2\Phi^3
\right|.
\end{aligned}
\end{equation}
By Lemma \ref{properties of N',Q'}-(4), \eqref{3.6} and Young's inequality, we have
\begin{equation}\label{J1.2.2.1.1}
\begin{aligned}
\int_0^t\int_{\mathbb{R}}
\left|N_x(\Phi\Psi+\Phi Q+N\Psi+\Phi_x)
\Phi^3\right|
\leq CE(t)
\int_0^t\int_{\mathbb{R}}(|Q_x|\Phi^2
+\Phi_x^2).
\end{aligned}
\end{equation}
As in \eqref{I1.2.2.1.2}-\eqref{I1.2.2.1.3}, by the Sobolev inequality we get
\begin{equation}\label{J1.2.2.1.2}
\begin{aligned}
\int_0^t\int_{\mathbb{R}}
\left|\Phi_x\Phi^4\Psi+\Phi_x\Phi^4Q+
\Phi_xN\Psi\Phi^3\right|
\leq CE(t)
\int_0^t\int \Phi_x^2,
\end{aligned}
\end{equation}
and
\begin{equation}\label{J1.2.2.1.3}
\begin{aligned}
\int_0^t\int_{\mathbb{R}}
\left|\Phi_x\Phi^3\Psi+\Phi_x\Phi^3Q+
\Phi_xN\Psi\Phi^2+\Phi_x^2\Phi^2+\Phi_x^2\Phi^3
\right|
\leq CE(t)
\int_0^t\int \Phi_x^2.
\end{aligned}
\end{equation}
Substituting \eqref{J1.2.2.1.1}-\eqref{J1.2.2.1.3} into \eqref{J1.2.2.1'} and combining the result with \eqref{J1.2.2.3}, we obtain
\begin{equation*}
\begin{aligned}
\left|\int_0^t\int_{\mathbb{R}}
(N+\Phi)^{-2}\Phi_t\frac{\Phi^3}{6}\right|
\leq CE(t)
\int_0^t\int(|Q_x|\Phi^2
+\Phi_x^2+|\Phi (N_{xx}+g(N,Q)_x)|).
\end{aligned}
\end{equation*}
This together with \eqref{4.9}-\eqref{J1.2.1} leads to
\begin{equation}\label{J1}
\begin{aligned}
\int_0^t\int_{\mathbb{R}}
\frac{\Phi^2}{2}(N+\Phi)^{-1}\Phi_t
\leq &CE(t)\|\Phi(t)\|^2+CE(0)\|\Phi_0\|^2
+CE(t)
\int_0^t\int(|Q_x|\Phi^2
+\Phi_x^2)\\
&+CE(t)\int_0^t\int|\Phi (N_{xx}+g(N,Q)_x)|.
\end{aligned}
\end{equation}
Integrating by parts, by Lemma \ref{properties of N',Q'}-(4), we find
\begin{equation}\label{J2'}
\begin{aligned}
-\int_0^t\int_{\mathbb{R}}(N+\Phi)^{-1}
\frac{\Phi^2}{2}
Q\Phi_{x}dxd\tau
&=\int_0^t\int_{\mathbb{R}}
\partial_x
\left((N+\Phi)^{-1}Q\right)
\frac{\Phi^3}{6}dxd\tau\\
&\leq C
\int_0^t\int_{\mathbb{R}}|Q_x|
|\Phi|^3dxd\tau
+C\int_0^t
\|\Phi\|^2_{L^\infty}\|\Phi\|\|\Phi_x\|d\tau\\
&\leq C
\int_0^t\int_{\mathbb{R}}|Q_x|
|\Phi|^3dxd\tau
+C\int_0^t
\|\Phi\|^2\|\Phi_x\|^2d\tau\\
&\leq CE(t)
\int_0^t\int(|Q_x|\Phi^2
+\Phi_x^2),
\end{aligned}
\end{equation}
where we have used Sobolev inequality in the second inequality. Similarly, by Lemma \ref{properties of N',Q'}-(4), \eqref{Ux'} and \eqref{3.6}, we arrive at
\begin{equation}\label{J3-J5}
-\int_0^t\int_{\mathbb{R}}(N+\Phi)^{-1}
\frac{\Phi^2}{2}
\Psi\Phi_{x}dxd\tau
\leq CE(t)
\int_0^t\int\Phi_x^2,\end{equation}
\begin{equation}
-\int_0^t\int_{\mathbb{R}}(N+\Phi)^{-1}
\frac{\Phi^2}{2}
(N_x\Psi+\Phi Q_x)dxd\tau
\leq CE(t)
\int_0^t\int_{\mathbb{R}}|Q_x|\Phi^2dxd\tau,\end{equation}
and
\begin{equation}
\begin{split}
-\int_0^t\int_{\mathbb{R}}(N+\Phi)^{-1}
\frac{\Phi^2}{2}\Phi_{xx}dxd\tau
&=\int_0^t\int_{\mathbb{R}}
\partial_x
\left((N+\Phi)^{-1}\frac{\Phi^2}{2}\right)
\Phi_xdxd\tau \\&
\leq CE(t)
\int_0^t\int_{\mathbb{R}}(|Q_x|\Phi^2+\Phi_x^2).\end{split}
\end{equation}
Noting $N(x,t)+\Phi(x,t)>{\color{red}\frac{\bar n}{2}}$, by \eqref{3.6}, we have
\begin{equation}\label{J6-J7}
\begin{aligned}
-\int_0^t\int_{\mathbb{R}}(N+\Phi)^{-1}
\frac{\Phi^2}{2}(N_{xx}+g(N,Q)_x)dxd\tau
&\leq C\int_0^t\int_{\mathbb{R}}
|\Phi|^2|(N_{xx}+g(N,Q)_x)| dxd\tau\\
&\leq CE(t)\int_0^t\int_{\mathbb{R}}|\Phi (N_{xx}+g(N,Q)_x)|dxd\tau.
\end{aligned}
\end{equation}
Substituting \eqref{J1}-\eqref{J6-J7} into \eqref{Psix}, one has
\begin{equation*}
\begin{aligned}
\left|\int_0^t\int_{\mathbb{R}}\frac{\Phi^2}{2} \Psi_xdxd\tau\right|
\leq &CE(t)\|\Phi(t)\|^2+CE(0)\|\Phi_0\|^2
+CE(t)
\int_0^t\int_{\mathbb{R}}(|Q_x|\Phi^2+\Phi_x^2)\\
&+CE(t)\int_0^t\int_{\mathbb{R}}|\Phi (N_{xx}+g(N,Q)_x)|dxd\tau.
\end{aligned}
\end{equation*}
It then follows from \eqref{Phi2+Psi2} that
\begin{equation}\label{4.1'}
\begin{aligned}
&\|\Phi(t)\|^2
+\|\Psi(t)\|^2
+\int_0^t\int_{\mathbb{R}}|Q_{x}|\Phi^2
+\int_0^t\int_{\mathbb{R}}|N_{t}|\Psi^2
+\int_0^t\|\Phi_x(\tau)\|^2
\\ &\leq \fix{C(\|\Phi_0\|^2
+\|\Psi_0\|^2)}
+C\int_0^t\int_{\mathbb{R}}|\Phi N_{xx}|dxd\tau
+C\int_0^t\int_{\mathbb{R}}|\Phi g(N,Q)_x|dxd\tau\\&
\leq \fix{C(\|\Phi_0\|^2
+\|\Psi_0\|^2)}
+C\int_0^t
\|\Phi\|(\|N_{xx}\|+\|g(N,Q)_x\|)d\tau,
\end{aligned}
\end{equation}
provided that $E(t)$ is suitably small.
Taking in Lemma \ref{Gronwall}
\begin{equation*}
\begin{aligned}
F(t)\triangleq \fix{C(\|\Phi_0\|^2
+\|\Psi_0\|^2)}+C\int_0^t
\|\Phi\|(\|N_{xx}\|+\|g(N,Q)_x\|)d\tau
,\ \beta(t)\triangleq \|N_{xx}\|+\|g(N,Q)_x\|,
\end{aligned}
\end{equation*}
by \eqref{N'xx} and \eqref{tgx}, we obtain the estimate \eqref{L2'}.
The proof is complete.
\end{proof}

We next derive the estimate for $\|\Psi_x\|$ .
\begin{lemma}\label{H1Psix}
Let the assumptions of Proposition \ref{'A priori estimates'} hold.
If $\chi_1\ll1$, there exists a constant $C>0$ such that
\begin{equation}\label{H1.1'}
\begin{aligned}
\|\Psi_x(\cdot,t)\|^2
+\int_0^t\|
\Psi_x(\cdot,\tau)\|^2
\leq &C(\|\Phi_0\|^2
+\|\Psi_0\|^2
+\|\Psi_{0x}\|^2
+\varepsilon) \ \text{ for } t\in[0,T].
\end{aligned}
\end{equation}

\end{lemma}

\begin{proof}
Substituting $\Psi_t=\Phi_x$ into the first equation of \eqref{rewrite'} yields
\begin{equation}\label{Psitx}
\begin{aligned}
\Psi_{tx}=\Phi_t -(\Phi\Psi+\Phi Q+N\Psi)_x-N_{xx}-g(N,Q)_x.
\end{aligned}
\end{equation}
Multiplying \eqref{Psitx} by $\Psi_x$, and integrating the result, noting \begin{equation*}
\begin{aligned}
\Phi_t\Psi_x&=(\Phi\Psi_x)_t-\Phi\Psi_{xt}
=(\Phi\Psi_x)_t-\Phi\Phi_{xx}
=(\Phi\Psi_x)_t-(\Phi\Phi_x)_x+\Phi_x^2,
\end{aligned}
\end{equation*}
we have by integration by parts,
\begin{equation}\label{Psix2}
\begin{aligned}
&\left.\left(\int_{\mathbb{R}}\frac{\Psi_x^2}{2}
- \Phi\Psi_x dx\right)\right|_0^t
+\int_0^t\int_{\mathbb{R}}        N\Psi_x^2dxd\tau\\
=&\int_0^t\int_{\mathbb{R}}\Phi_x^2dxd\tau
-\int_0^t\int_{\mathbb{R}}\Psi_{x}(\Phi Q)_{x} dxd\tau
-\int_0^t\int_{\mathbb{R}}\psi_{x}(\Phi\Psi)_{x} dxd\tau\\
&+\int_0^t\int_{\mathbb{R}}\frac{\Psi^2}{2}N_{xx} dxd\tau
-\int_0^t\int_{\mathbb{R}}\Psi_{x}N_{xx} dxd\tau
-\int_0^t\int_{\mathbb{R}}\Psi_{x}g(N,Q)_x dxd\tau.
\end{aligned}
\end{equation}
As in the proof of Lemma \ref{H1psix}, we obtain
\begin{equation}\label{Psix2'}
\begin{aligned}
&\|\Psi_x(t)\|^2
+\int_0^t\|
\Psi_x(\tau)\|^2\\
&\leq C(\|\Phi_0\|^2
+\|\Psi_0\|^2+\|\Psi_{0x}\|^2)
+C E\varepsilon
+\int_0^t\int_{\mathbb{R}}
|\Psi_x(N_{xx}+g(N,Q)_x )| dxd\tau\\
&\leq C(\|\Phi_0\|^2
+\|\Psi_0\|^2+\|\Psi_{0x}\|^2)
+C \varepsilon
+C\int_0^t
\|\Psi_x\|(\|N_{xx}\|+\|g(N,Q)_x\|)d\tau,
\end{aligned}
\end{equation}
provided that $E(t)\ll1$ is suitably small.
By Lemma \ref{Gronwall}, \eqref{N'xx} and \eqref{tgx}, we obtain the estimate \eqref{H1.1'}.
\end{proof}

Finally, we establish the estimate for $\|\Phi_x\|$ .
\begin{lemma}\label{H1Phix}
Let the assumptions of Proposition \ref{'A priori estimates'} hold.
If $\chi_1\ll1$, then there exists a constant $C>0$ such that
\begin{equation}\label{H'1.2}
\begin{aligned}
\|\Phi_x(\cdot,t)\|^2
+\int_0^t\|\Phi_{xx}(\cdot,\tau)\|^2
\leq &C(\|\Phi_0\|_1^2
+\|\Psi_0\|_1^2
+\varepsilon) \text{ for } t\in[0,T].
\end{aligned}
\end{equation}

\end{lemma}

\begin{proof}
Multiplying the first equation of \eqref{rewrite'} by $-\Phi_{xx}$ and integrating the result, we have
\begin{equation}\label{Phixx}
\begin{aligned}
&\left.\int_{\mathbb{R}}\frac{\Phi_x^2}{2}
dx\right|_0^t
+\int_0^t\int_{\mathbb{R}}        \Phi_{xx}^2dxd\tau\\
=&-\int_0^t\int_{\mathbb{R}}
\Phi_{xx}\left[(\Phi\Psi+\Phi Q+N\Psi)_{x}
+N_{xx}+g(N,Q)_x
\right]dxd\tau\\
\leq &\frac{1}{2}\int_0^t\int_{\mathbb{R}}
\Phi_{xx}^2 dxd\tau
+\frac{1}{2}\int_0^t\int_{\mathbb{R}}
|N_{xx}|^2dxd\tau
+C\int_0^t\int_{\mathbb{R}}
(|Q||\Phi_x|^2+|Q_{x}|^2|\Phi|^2+|N_{x}|^2|\Psi|^2
\\&\quad+|N||\Psi_x|^2
+|\Phi_x|^2|\Psi|^2+|\Phi|^2|\Psi_x|^2
)dxd\tau
+\frac{1}{2}\int_0^t\int_{\mathbb{R}}
|g(N,Q)_x|^2dxd\tau.
\end{aligned}
\end{equation}
As in the proof of Lemma \ref{H1phix}, we get
\begin{equation}\label{H}
\begin{aligned}
&\int_0^t\int_{\mathbb{R}}
|Q_{x}|^2|\Phi|^2dxd\tau
\leq C\int_{\mathbb{R}}
|Q_{x}||\Phi|^2dxd\tau,\\
&\int_0^t\int_{\mathbb{R}}
|N_{x}|^2|\Psi|^2dxd\tau
\leq {\color{red}C\varepsilon\int_0^t\int_{\mathbb R}|N_t||\Psi|^2dxd\tau,}\\
&\int_0^t\int_{\mathbb{R}}
|\Phi_x|^2|\Psi|^2dxd\tau
\leq CE^2\int_0^t
\|\Phi_x\|^2d\tau,\\
&\int_0^t\int_{\mathbb{R}}
|\Phi|^2|\psi_x|^2dxd\tau
\leq CE^2\int_0^t
\|\Psi_x\|^2d\tau,
\end{aligned}
\end{equation}
\fix{where we have used the structural bounds \eqref{composite-struct}: $\|Q_x\|_{L^\infty}\leq C\varepsilon$ and $|N_x|^2\leq|N_t|\,|Q_x|$ in the second inequality.}
By {\color{red}the second estimate in \eqref{gx}}, the last term of \eqref{Phixx} satisfies
\begin{equation}\label{gx2}
\begin{aligned}
\frac{1}{2}\int_0^t\int_{\mathbb{R}}
|g(N,Q)_x|^2dxd\tau
\leq C_{\theta}\int_0^t
\varepsilon^3
(1+(\varepsilon \tau)^2)^{-2\theta/3}d\tau
\leq {\color{red} C}\varepsilon^2.
\end{aligned}
\end{equation}
Substituting \eqref{H} and \eqref{gx2} into \eqref{Phixx}, by
Lemmas \ref{L2 estimate'} and \ref{H1Psix}, we obtain the desired estimate \eqref{H'1.2} and  finish the proof of Lemma \ref{H1Phix}.
\end{proof}

\begin{proof}[Proof of Proposition \ref{'A priori estimates'}]
It is a direct consequence of Lemmas \ref{L2 estimate'}, \ref{H1Psix} and \ref{H1Phix}.
\end{proof}

\begin{proof}[Proof of Proposition \ref{Global existence}]
\fix{For the fixed admissible end states, let $\inf_{x,t}N(x,t)>0$. Choose $\delta_1$ and $\varepsilon$ so small that the $H^1$ bound in \eqref{global estimates} strictly improves the bootstrap condition $E(T)\le\chi_1$ and yields $N+\Phi\ge \frac{\inf_{x,t}N(x,t)}{2}$ by Sobolev embedding. Starting with the local solution and applying the a priori bound until its maximal time, the standard positive-density continuation criterion then gives a global solution. In particular, the argument uses the forcing-dependent local theory with $N_{xx}+g(N,Q)_x$, rather than a homogeneous local estimate.}
By the estimate \eqref{global estimates} and the equation \eqref{rewrite'}, applying the same argument as that of Proposition \ref{phi stability}, we have
\begin{equation*}
\begin{aligned}
\int_{0}^{+\infty}\left[\int_{\mathbb{R}}
\left(\Phi_x^2+\Psi_x^2\right)dx+
\left|\frac{d}{dt}\int_{\mathbb{R}}
\left(\Phi_x^2+\Psi_x^2\right)dx\right|\right]d\tau
<+\infty.
\end{aligned}
\end{equation*}
Hence
\begin{equation*}
\lim\limits_{t\rightarrow \infty}
\left\|(\Phi_x,\Psi_x)(\cdot,t)\right\|^2=0.
\end{equation*}
It then follows from the Sobolev inequality  that
\begin{equation}\label{4.35}
\begin{split}
\sup\limits_{x\in \mathbb{R}}
\left|(\Phi,\Psi)(x,t)\right|^2&\leq C\max\{\left\|\Phi(\cdot,t)\right\|\left\|\Phi_x(\cdot,t)\right\|,\left\|\Psi(\cdot,t)\right\|\left\|\Psi_x(\cdot,t)\right\|\}\\&\leq C\max\{\left\|\Phi_x(\cdot,t)\right\|,\left\|\Psi_x(\cdot,t)\right\|\}\rightarrow0 \text{ as }t\rightarrow\infty.
\end{split}\end{equation}
This completes the proof of  Proposition \ref{Global existence}.
\end{proof}

We are ready to finish the proof of Theorem \ref{the1.2}.
\begin{proof}[Proof of Theorem \ref{the1.2}]
Since $(n-N,q-Q)=(\Phi,\Psi)$, if $(n_0-N_0,q_0-Q_0)(x)$ is small, then by Proposition \ref{Global existence}, the system \eqref{1.1}-\eqref{i d} has a unique global solution satisfying \eqref{1.9}. Furthermore, by \eqref{limN} and \eqref{4.35}, $(n,q)$ has the asymptotic behavior \eqref{stable}. We complete the proof of Theorem \ref{the1.2}.
\end{proof}

\section*{Acknowledgements}
This work is supported by the National Natural Science Foundation of China (No. 12371216).

\end{document}